\documentclass[11pt]{amsart}

\usepackage{multirow}
\usepackage{amssymb}
\usepackage{csvsimple}
\usepackage{graphicx,color}
\graphicspath{{figs/}}
\def\R{\mathbb{R}}

\def\cI{\mathcal{I}}

\def\cO{\mathcal{O}}

\def\cT{\mathcal{T}}

\def\a{\alpha}
\def\b{\beta}
\def\g{\gamma}
\def\d{\delta}

\def\k{\kappa}

\def\p{\partial}

\def\veps{\varepsilon}

\def\vphi{\varphi}
\def\O{\Omega}

\def\GD{{\Gamma_{\rm D}}}

\def\DD{{\rm D}}

\def\tu{\widetilde{u}}

\newcommand{\dv}[1]{\,{\mathrm d}#1}

\newcommand{\wcheck}[1]{#1\hspace{-.8ex}\mbox{\huge {\lower.45ex \hbox{$\textstyle \check{}$}}} \hspace{.5ex}}

\DeclareMathOperator{\diver}{div}

\let\oldmarginpar\marginpar
\renewcommand\marginpar[1]{
  \oldmarginpar[\raggedleft\footnotesize #1]
  {\raggedright\footnotesize #1}}
\usepackage{ifthen}
\usepackage[normalem]{ulem}

\newtheorem{definition}{Definition}
\newtheorem{lemma}[definition]{Lemma}
\newtheorem{proposition}[definition]{Proposition}

\newtheorem{corollary}[definition]{Corollary}
\newtheorem{remark}[definition]{Remark}

\newtheorem{example}[definition]{Example}
\newtheorem{examples}[definition]{Examples}
\newtheorem{algorithm}[definition]{Algorithm}
\numberwithin{definition}{section}

\definecolor{tourquoise}{RGB}{0,170,180}	
\definecolor{darkred}{RGB}{238,34,34}		
\definecolor{darkgreen}{RGB}{0,190,0}		
\definecolor{lightgray}{RGB}{210,210,210}	
\definecolor{deepblue}{RGB}{0,0,240}		
\definecolor{darkgray}{RGB}{144,144,144}	
\definecolor{kingblue}{RGB}{64,96,224}		
\definecolor{gold}{RGB}{240,208,0}		
\definecolor{verydarkred}{RGB}{176,0,0}		

\usepackage{boxedminipage}
\def\tu{\widetilde{u}}
\def\gek{\g_{\veps,k}}
\def\pO{{\partial \O}}
\usepackage{stmaryrd}
\begin{document}
\title[Stability of semi-implicit schemes]{Unconditional 
stability of semi-implicit discretizations of singular flows}
\author[S. Bartels]{S\"oren Bartels}
\address{Department of Applied Mathematics, 
Albert Ludwigs University Freiburg, Germany.}
\email{bartels@mathematik.uni-freiburg.de}
\author[L. Diening]{Lars Diening}
\address{Department of Mathematics, University of Bielefeld, Germany}
\email{lars.diening@uni-bielefeld.de}
\author[R. H. Nochetto]{Ricardo H. Nochetto}
\address{Department of Mathematics and Institute for Physical Science
and Technology, University of Maryland, College Park, MD.}
\email{rhn@math.umd.edu}
\date{\today}



\begin{abstract}
A popular and efficient discretization of evolutions involving the 
singular $p$-Laplace operator is based on a factorization of the
differential operator into a linear part which is treated implicitly
and a regularized singular factor which is treated explicitly. It 
is shown that an unconditional energy stability property for
this semi-implicit time stepping strategy holds. Related error estimates 
depend critically on a required regularization parameter.
Numerical experiments reveal reduced experimental convergence rates
for smaller regularization parameters and 
thereby confirm that this dependence cannot be avoided in general.
\end{abstract}

\maketitle

\section{Introduction}
We discuss the numerical solution of minimization and evolution
problems related to the $p$-Dirichlet energy 
\[
E_p[u] = \frac1p \int_\O |\nabla u|^p\dv{x},
\]
with $1\le p < 2$. The Euler--Lagrange equations give rise to a singular
differential operator which requires a careful numerical treatment. 
Related problems occur in the description of minimal surfaces, 
porous media, non-Newtonian fluids, nonlinear elasticity,
and Newton's problem of minimal resistance; we refer the reader 
to~\cite{Dziu99,Cham04,DeDzEl05,FeOePr05,BeDiRu15,DiFoWa17-pre} for related 
results. Typically, standard numerical schemes such as Newton or 
Picard iterations fail to determine stationary configurations. 

Gradient flows provide a robust approach to find minimizers for 
functionals that involve $E_p$ or arise as models to describe 
certain nonlinear evolutions. In the simplest case this leads to
the equation 
\begin{equation}\label{eq:p_flow}
\p_t u - \diver \big(|\nabla u|^{p-2} \nabla u\big) = 0,
\end{equation}
subject to initial and boundary conditions. An implicit discretization
in time leads to the nonlinear recursion formula 
\begin{equation}\label{nonlinear-recursion}
d_t \tu^k - \diver \big(|\nabla \tu^k|^{p-2} \nabla \tu^k\big) = 0,
\end{equation}
for $k=1,2,\dots,K$, with a step-size $\tau>0$ and the backward 
difference quotient operator $d_t c^k = (c^k-c^{k-1})/\tau$. The
iterates $(\tu^k)_{k=0,\dots,K}$ are well defined and optimal 
error estimates
\[
\max_{k=0,\dots,K} \|u(t_k) - \tu^k\| = \cO(\tau),
\]
with $t_k = k \tau$, can be derived under appropriate conditions 
on the initial function~$u^0$, 
cf.~\cite{BarLiu93,BarLiu94,Rull96,NoSaVe00,DiEbRu07}.

Unfortunately, the development of efficient numerical schemes
for computing the iterates $(\tu^k)_{k=0,\dots,K}$ is far from 
being obvious. Moreover, including perturbation terms in the error
analysis of the implicit scheme shows that very restrictive stopping
criteria for the iterative approximate solution are necessary. 
It is therefore desirable to develop time-discretizations that lead to
linear systems of equations in every time step but still have good
stability properties. In fact, such schemes can then also be used as
iterative solvers for approximating nonlinear problems
such as \eqref{nonlinear-recursion}.

A popular approach to discretizing the nonlinear partial 
differential equation consists in defining iterates $(u^k)_{k=0,\dots,K}$
via a semi-implicit discretization of~\eqref{eq:p_flow}
and the corresponding sequence of linear problems
\begin{equation}\label{eq:simpl}
d_t u^k - \diver \big(|\nabla u^{k-1}|_\veps^{p-2} \nabla u^k \big) = 0,
\end{equation}
for $k=1,2,\dots,K$. Here, the use of a regularization of the
euclidean length, e.g., defined via $|a|_\veps = (a^2+\veps^2)^{1/2}$
with a positive parameter $\veps$, guarantees that the iterates
are well-defined. The unconditional well-posedness in the sense
of stability of the iteration is nonobvious due to the loss of 
monotonicity properties related to the implicit-explicit treatment
of the differential operator. 
It is the purpose of this article to demonstrate that the
iteration is nonetheless unconditonally energy stable and to provide
error estimates that control the influence of the regularization
and semi-implicit discretization on the quality of approximations. 
A related stability estimate has been proved for the mean 
curvature flow of graphs in~\cite{Dziu99} which corresponds to the 
case $p=1$ and $\veps = 1$.

We discuss now our unexpected observation for the special and most 
singular situation corresponding to the exponent $p=1$, the so-called
regularized {\em total variation flow}. Testing
the iterative scheme~\eqref{eq:simpl} with $d_t u^k$ and incorporating
a standard binomial formula leads to the identity
\[
\|d_t u^k\|^2  + \frac12 \int_\O 
\frac{d_t |\nabla u^k|^2 + \tau |d_t \nabla u^k|^2}{|\nabla u^{k-1}|_\veps} \dv{x} = 0.
\]
To identify the regularized energy $E_{1,\veps}$ on the left-hand side 
we employ difference quotient calculus and derive the formula
\[\begin{split}
d_t |a^k|_\veps  = d_t \frac{|a^k|_\veps^2}{|a^k|_\veps} 
& = \frac{d_t |a^k|_\veps^2}{|a^{k-1}|_\veps} + |a^k|_\veps^2 \, d_t \frac{1}{|a^k|_\veps} \\
& 
= \frac{d_t |a^k|_\veps^2}{|a^{k-1}|_\veps} - \frac{|a^k|_\veps d_t|a^k|_\veps}{|a^{k-1}|_\veps} \\
&= \frac{d_t |a^k|_\veps^2}{|a^{k-1}|_\veps} -  \frac12 \frac{d_t |a^k|_\veps^2 + \tau (d_t |a^k|_\veps)^2}{|a^{k-1}|_\veps} \\
&= \frac12 \frac{d_t |a^k|_\veps^2}{|a^{k-1}|_\veps} - \frac12 \frac{\tau (d_t |a^k|_\veps)^2}{|a^{k-1}|_\veps}.
\end{split}\]
Using this formula with $a^k = \nabla u^k$ and noting that $d_t |a^k|_\veps^2 = d_t |a^k|^2$ for the 
regularized euclidean length specified above, we find that
\[
\|d_t u^k\|^2 
+ d_t \int_\O |\nabla u^k|_\veps \dv{x} 
+ \frac{\tau}{2} \int_\O \frac{|d_t \nabla u^k|^2 + (d_t |\nabla u^k|_\veps)^2}{|\nabla u^{k-1}|_\veps} \dv{x} = 0.
\]
The last term on the left-hand side is nonnegative so that upon summation
over $k=1,2,\dots, L\le K$ and multiplication by $\tau$ we have the energy decay 
and unconditional stability property
\begin{equation}\label{unconditional-stab}
E_{1,\veps}[u^L] + \tau \sum_{k=1}^L \|d_t u^k\|^2 \le E_{1,\veps}[u^0],
\end{equation}
where $E_{1,\veps}$ results from replacing the euclidean length in $E_p$ with $p=1$ 
by a regularization. We will prove this inequality for a class of Orlicz type
functionals which includes the regularized $p$-Dirichlet energy as a special 
case. The arguments and the unconditional stability
estimate carry over verbatim to spatial discretizations
of the semi-implicit scheme. 

Good stability properties of a numerical scheme are important to obtain
useful error estimates. We derive bounds on the approximation error 
by controlling the differences between the iterates of the implicit 
and semi-implicit schemes and incorporating known error estimates for
the implicit discretizations. In contrast to the estimates for implicit
schemes we thereby obtain error estimates that involve a dependence on 
negative powers of the regularization parameter $\veps$. 
Moreover, we have to employ inverse estimates that introduce a critical dependence on the
spatial mesh-size~$h$. For lowest order continuous finite elements we obtain
the following error estimates
for the difference between the solution~$u$ of the gradient
flow \eqref{eq:p_flow} and
the approximations $(u_h^k)_{k=0,\dots,K}$ of the regularized,
semi-implicit scheme \eqref{eq:simpl}
\[\begin{split}
\max_{k=0,\dots,K} \|u(t_k) - u_h^k \| \le & \, c_{{\rm isf}} \tau^\a + 2 (c_{p,{\rm r}}T)^{1/2} \veps^{p/2} \\
& +
\begin{cases}
c_{1,{\rm i}} h^{\b} + c_{1,{\rm s}} \big(\tau h^{-2}
\veps^{-1}\big)^{1/2} & \mbox{for } p=1, \\
c_{p,{\rm i}} h^{\gamma} + c_{p,{\rm s}} \big(\tau h^{p-2} \veps^{p-2} \big)^{1/2} & \mbox{for } p>1,
\end{cases}
\end{split}\]
where $\b = 1/6$ or $1/4$ and $\gamma = 1-d(2-p)/8$.
The first term on the right-hand side results from the general analysis of implicit 
time discretization of subgradient flows, cf.~\cite{Rull96,NoSaVe00}; we 
have $1\le \a \le 2$, depending on regularity properties of the initial data. 
The second term accounts for the regularization of the evolution problem.
Spatial discretization errors due to the implicit scheme \eqref{nonlinear-recursion}
result in the first terms involving the positive powers of mesh-size $h$ under the 
case distinction. We observe a significant gap between the cases $p=1$ and $p>1$ which
is related to the fact that for $p>1$ regularity results for nonlinear parabolic
partial differential 
equations can be used, cf.~\cite{DiEbRu07}, while the analysis of the case $p=1$ is 
solely based on energy arguments using the limited regularity properties of solutions
provided by the problem, cf.~\cite{BaNoSa14,BaNoSa15}.
The exponent $\gamma = 1/6$ is generic while $\gamma = 1/4$ relies on
a total variation diminishing interpolation operator, which
is constructed in \cite{BaNoSa15} for special meshes and definition of
total variation using the $\ell^1$-norm for vectors.
We note that the constant $c_{p,{\rm i}}$ 
is expected and in fact has to deteriorate as $p\searrow 1$. The factor $h^{\b}$
can be replaced by $h$ if the reverse step-size condition $\tau \ge c h^{\a(p,d)}$ is
imposed. In our situation such a condition conflicts
with the last terms that involve the 
inverse of the mesh size. These terms result from the
semi-implicit time discretization \eqref{eq:simpl},
and here the gap between the two cases is related
to the strong monotonicity properties of the problem for $p>1$.

The outline of this article is as follows. In Section~\ref{sec:prelim} we 
specify notation and collect some basic estimates. Section~\ref{sec:stab}
is devoted to the generalization of the unconditional stability estimate
for semi-implicit discretizations of a class of singular flows
including \eqref{eq:p_flow}. An error analysis
for fully discrete schemes is provided in Section~\ref{sec:error}. 
Numerical experiments for the case $p=1$ illustrate our theoretical results
and are presented in Section~\ref{sec:numex}. 

\section{Preliminaries}\label{sec:prelim}
\subsection{Notation}
We use standard notation for Lebesgue and Sobolev spaces on the bounded
Lipschitz domain $\O\subset \R^d$. The inner product on $L^2(\O;\R^\ell)$
is denoted by $(\cdot,\cdot)$ and the corresponding norm by $\|\cdot\|$. 
For a closed, possibly empty subset $\GD\subset \p\O$ we let $W^{1,p}_\DD(\O)$
be the set of functions in $W^{1,p}(\O)$ that vanish on $\GD$; we write
$W^{1,p}_0(\O)$ if $\GD = \p\O$. 
The space $BV(\O)$ consists of all functions $v\in L^1(\O)$ with bounded
total variation, i.e., functions $v\in L^1(\O)$ with 
\begin{equation}\label{tv-norm}
|Dv|(\O) = \sup_{\xi \in C_0^\infty(\O;\R^d), \mbox{ }\|\xi\|_{L^\infty(\O)} \le 1} 
- \int_\O v \diver \xi \dv{x} < \infty.
\end{equation}
For a shape regular triangulation $\cT_h$ of the polyhedral domain~$\O$ into
simplices, we let 
\[
V_h = \big\{v_h\in C(\overline{\O}): v_h|_T \in P_1(T) \text{ for all } T\in \cT_h\big\},
\]
be the space of piecewise affine, continuous finite element functions on $\cT_h$.
The parameter $h>0$ represents the maximal mesh-size of the triangulation. 

\subsection{Difference calculus}
Given a sequence $(c^k)_{k=0,\dots,K}$ and a step size $\tau>0$ we
define the backward difference quotient via
\[
d_t c^k = \frac{1}{\tau} \big(c^k- c^{k-1}\big)
\]
for $k=1,2,\dots,K$. We note the discrete product and quotient rules
\[\begin{split}
d_t \big(c^k\cdot b^k\big) 
& = \big(d_t c^k\big) \cdot b^{k-1} + c^k \cdot \big(d_t b^k \big), \\
d_t \big(1/c_k\big) &= -d_t c^k / \big(c^{k-1}c^k\big).
\end{split}\]
Moreover, we have the identity
\begin{equation}\label{eq:prod}
c^k \cdot d_t c^k 
= \frac12 d_t \big|c^k\big|^2 + \frac{\tau}{2} \big|d_t c^k\big|^2.
\end{equation}
They have been used earlier in deriving \eqref{unconditional-stab}.

\subsection{Regularized euclidean length}
We consider a family of regularizations
$|\cdot|_\veps$, $\veps\in [0,\veps_0]$, of the euclidean length $|\cdot|$ 
such that for $\veps>0$ the mapping
\[
|\cdot|_\veps : \R^d \to \R_{\ge 0}
\]
is continuously differentiable and convex. We
assume that we have the estimate
\begin{equation}\label{eq:approx_mod}
\big||a|_\veps^p - |a|^p \big|\le c_{p,{\rm r}} \, \veps^p
\end{equation}
for all $a\in \R^d$ with a constant $c_{p,{\rm r}}>0$ that may depend 
on $1\le p < 2$. 

\begin{examples}
(i) For the {\em standard regularization} $|a|_\veps = (|a|^2+\veps^2)^{1/2}$ we
have for $a\in \R^d$ with $|a|= s \veps$ that 
\[
|a|_\veps^{p} - |a|^p 
= \big((s^2+1)^{p/2} - (s^2)^{p/2} \big) \veps^p = f(s^2) \veps^p \le \veps^p,
\]
since $f(r) = (r+1)^{p/2} - r^{p/2}$ is monotonically decreasing
with $f(0)=1$.  \\
(ii) The {\em truncated regularization} defined for $a\in \R^d$ and 
$\veps\ge 0$ via 
\[
|a|_\veps^p = \begin{cases}
|a|^p + (p/2-1)\veps^p & \mbox{for } |a| \ge \veps, \\
(p/2) \veps^{p-2} |a|^2 & \mbox{for } |a| \le \veps,
\end{cases}
\]
satisfies~\eqref{eq:approx_mod} with $c_{p,{\rm r}} = (2-p)/2$.
\end{examples}

\subsection{Subgradient flow and regularization}
We interpret the nonlinear evolution equation~\eqref{eq:p_flow} as 
a subgradient flow for the possibly regularized $p$-Dirichlet energy 
\[
E_{p,\veps}[u] = \frac1p \int_\O |\nabla u|_\veps^p \dv{x},
\]
for $u\in X$ with $X=W^{1,p}_\DD(\O)$. If $p=1$ and $\veps=0$ we define 
$E_{p,\veps}[u]$ as the total variation \eqref{tv-norm}
of $u$ and choose $X=BV(\O)$. 
The functionals $E_{p,\veps}$ are formally extended to $L^2(\O)$ by
assigning the value $+\infty$ to $u\in L^2(\O)\setminus X$. 
The existence of a unique function $u\in W^{1,2}([0,T];L^2(\O)) \cap L^\infty([0,T];X)$ 
which satisfies $u(0) = u^0$ continuously for a given $u^0\in L^2(\O)\cap X$ and  
\begin{equation}\label{eq:subflow}
-\p_t u \in \p E_{p,\veps}[u],
\end{equation}
for almost every $t\in (0,T)$ is well established 
for all $\veps\ge0$, cf.~\cite{Brez73}. 
Note that we always consider the subdifferential with respect to the 
$L^2$ scalar product, i.e., 
\[
\p E_{p,\veps}[u] = \big\{ s\in L^2(\O): (s,v-u)+ E_{p,\veps}[u] \le E_{p,\veps}[v]
\mbox{ for all } v \in L^2(\O) \big\}.
\]
We thus have that the inclusion \eqref{eq:subflow}
is equivalent to the variational inequality
\[
(-\p_t u, v-u) + E_{p,\veps}[u] \le E_{p,\veps}[v],
\]
for all $v\in L^2(\O)$ and $\veps\ge0$. For $\veps>0$, \eqref{eq:subflow} is also
equivalent to the equation
\begin{equation}\label{eq:reggradflow}
(\p_t u,v) + (|\nabla u|_\veps^{p-2} \nabla u,\nabla v) = 0,
\end{equation}
for all $v\in X$ and $t\in (0,T)$.
Letting $u$ and $u_\veps$ be the solutions
of the subgradient flows for a fixed $p\in [1,2)$, subject to the same initial
condition, and $\veps=0$ and $\veps>0$, respectively, we deduce 
from~\eqref{eq:approx_mod} via straightforward calculations that 
\[
\sup_{t\in [0,T]} \|u-u_\veps\| \le 2(c_{p,{\rm r}} T)^{1/2} \veps^{p/2}.
\]

\subsection{Implicit time discretization}
Given a time step $\tau>0$, stable approximations of the solution of the 
subgradient flow~\eqref{eq:subflow} are defined by the implicit Euler scheme 
\[
\tu^k = \mbox{argmin}_{v\in X} \, \frac{1}{2\tau} \|v-\tu^{k-1}\|^2 + E_{p,\veps}[v],
\] 
for $k=1,2,\dots,K$, initialized with $\tu^0 = u^0$. 
The sequence $(\tu^k)_{k=0,\dots,K}$ is uniquely defined and the
iterates satisfy 
\[
(-d_t \tu^k, v-\tu^k) + E_{p,\veps}[\tu^k] \le E_{p,\veps}[v],
\]
for all $v\in X$. We have the error estimate, cf.~\cite{Rull96,NoSaVe00}, 
\[
\max_{k=0,\dots,K} \|u(t_k) - \tu^k \|\le c_{{\rm isf}} \tau^\a,
\]
with $\a = 1/2$ if $E_{p,\veps}[u^0]< \infty$ and 
$\a=1$ if $\p E_{p,\veps}[u^0]\neq \emptyset$. 

\subsection{Spatial discretization}
A spatial discretization of the implicit time stepping scheme for the
subgradient flow 
determines iterates $(\tu_h^k)_{k=0,\dots,K} \subset X_h$
with $X_h = V_h \cap X$ 
for a suitable approximation $\tu_h^0$ of $u^0$ via the sequence of minimization
problems
\[
\tu_h^k = \mbox{argmin}_{v_h\in X_h} \frac{1}{2\tau} \|v_h - \tu_h^{k-1} \|^2 
+ E_{p,\veps}[v_h].
\]
Invoking~\cite{BaNoSa14,BaNoSa15,Bart15-book}
for the case $p=1$ and~\cite{DiEbRu07} for the case $p>1$, we have the error estimates 
\begin{equation}\label{eq:p_flow_est}
\max_{k=0,\dots,K} \|u(t_k)- \tu_h^k\| \le c_{{\rm isf}} \tau^\a + 2(c_{p,{\rm r}}T)^{1/2} \veps^{p/2}  +
\begin{cases}
  c_{1,{\rm i}} h^{\b} & \mbox{for } p=1, \\
  c_{p,{\rm i}} h^{\gamma} & \mbox{for } p>1,
\end{cases}
\end{equation}
for suitable choices of $\tu_h^0$ and with $\b = 1/6$ or $1/4$ and $\gamma=1-d(2-p)/8$.
The estimate of~\cite{BaNoSa14} for $p=1$ 
assumes homogeneous Neumann boundary conditions,
that $\O$ is star-shaped, and that $u^0\in BV(\O) \cap L^\infty(\O)$,
and holds with $\a = 1/2$ and $\b = 1/6$. The decay rate in space can be
improved to $\beta = 1/4$ upon utilizing a total variation
diminishing interpolation operator, whose construction is discussed
in \cite{BaNoSa15} for special cartesian meshes and definition of the
total variation in terms of $\ell^1$-norms of vectors.
On the other hand,
the estimate of~\cite{DiEbRu07} for $p\in (1,2)$ 
assumes homogeneous Dirichlet boundary conditions,
that $\O$ is convex, and that the initial value satisfies $u^0\in W^{1,2}_0(\O)$
and $\diver \big(|\nabla u^0|^{p-2}\nabla u^0)\big)\in L^2(\O)$. This result 
entails the condition $p>2d/(d+2)$ which can be omitted  
when $\p_t u$ is an admissible test function, i.e., in 
case of subgradient flows and smooth right-hand sides.
Note that the assumptions on $u^0$ imply $\p E_{p,\veps}[u_0]\neq \emptyset$ so that
we may choose $\a=1$ \cite{Rull96,NoSaVe00,DiEbRu07}.

We remark that in the error estimate~\eqref{eq:p_flow_est} the function~$u$ may be
replaced by the solution~$u_\veps$ of the regularized evolution equation 
in which case the term involving the factor $\veps^{p/2}$ can be omitted. 

\section{Generalized unconditional stability estimate}\label{sec:stab}
We next generalize our unconditional stability estimate for 
semi-implicit discretizations to a class of gradient flows for convex
energy functionals $E_\vphi: L^2(\O) \to \R\cup\{+\infty\}$ 
defined with functions $\vphi: \R_{\ge 0}\to \R_{\ge 0}$ via
\[
E_\vphi[u] = \int_\O \vphi(|\nabla u|) \dv{x}.
\]
We impose the following conditions on the energy density $\vphi$
which define a class of sub-quadratic Orlicz functions:
\begin{itemize}
\item[(C1)] $r\mapsto \vphi(r)$ is convex and continuously differentiable with $\vphi(0)=0$,
\item[(C2)] $r\mapsto \vphi'(r)/r$ is positive, nonincreasing, and continuous on $\R_{\ge 0}$.
\end{itemize} 
Condition~(C2) implies that the following semi-implicit time-stepping scheme is
well posed. 

\begin{algorithm}[Semi-implicit scheme]\label{alg:simpl}
Let $u^0 \in X$ and $\tau,\veps>0$; set $k=1$. \\
(1) Compute $u^k\in X$ such that for all $v\in X$ we have
\[
(d_t u^k,v) 
+ \Big(\frac{\vphi'(|\nabla u^{k-1}|)}{|\nabla u^{k-1}|}\nabla u^k,\nabla v \Big) = 0.
\]
(2) Stop if $(k+1)\tau > T$; otherwise increase $k\to k+1$ and continue
with~(1).
\end{algorithm}

The regularized $p$-Dirichlet energy occurs as a special case of
(C1) and (C2).

\begin{examples}\label{ex:orlicz_fns}
(i) The regularized $p$-Laplace gradient flow corresponds to the function 
\[
\vphi(r) = \frac1p |r|_\veps^p - \frac1p |0|_\veps^p,
\]
and we have  
\[
\vphi'(r) = |r|_\veps^{p-2} r \quad \mbox{and} \quad
\vphi'(r) = \max\{\veps,r\}^{p-2} r,
\]
in case of the standard and truncated regularizations of euclidean length, respectively.
In both cases~(C1) and (C2) are satisfied for $1\le p < 2$.
A particular feature of the truncated regularization is that a closed formula
for the convex conjugate of $\vphi(|a|) = (1/p)|a|_\veps^p$ is
available.

\medskip
(ii) The function $\vphi(r) = r \ln(e+r)$ occurs in the modeling of Prandtl--Eyring fluids
and satisfies conditions (C1) and (C2), cf.~\cite{Eyri36,BrDiFu12} for details. 
\end{examples}

We remark that a positive $\veps$ is only needed for well-posedness of the semi-implicit
iteration of Algorithm \ref{alg:simpl}. Its
unconditional stability is a consequence of an elementary lemma. 

\begin{lemma}\label{la:orlicz_stab}
Under condition $(C2)$ we have for all $a,b\in \R^d$ that
\[
\frac{\vphi'(|a|)}{|a|} b \cdot (b-a) \ge \vphi(|b|)- \vphi(|a|) 
+ \frac12 \frac{\vphi'(|a|)}{|a|} |b-a|^2. 
\]
\end{lemma}

\begin{proof}
Using the identity $2 b\cdot (b-a) =  |b|^2 -|a|^2 + |b-a|^2$ we note that
\[
\frac{\vphi'(|a|)}{|a|} b\cdot (b-a) 
= \frac12 \frac{\vphi'(|a|)}{|a|} \big(|b|^2-|a|^2\big) 
 + \frac12 \frac{\vphi'(|a|)}{|a|} |b-a|^2.
\]
Since $r\mapsto \vphi'(r)/r$ is nonincreasing, 
the function $\psi(y) = \vphi(y^{1/2})$ is concave on $\R_{\ge 0}$,
so that we have 
\[
\psi'(y) (z-y) \ge \psi(z)-\psi(y),
\]
for all $y,z\ge 0$. With $y=|a|^2$ and $z=|b|^2$ we deduce that 
\[
\frac12 \frac{\vphi'(|a|)}{|a|} \big(|b|^2-|a|^2\big) 
\ge \vphi(|b|) -\vphi(|a|).
\]
Combining these inequalities implies the asserted estimate.
\end{proof}

The following proposition states the general unconditional stability
estimate for energy functionals $E_\vphi$ under conditions (C1) and (C2).
The estimate provides control over certain dissipation terms which will
be needed for the error estimates derived in the subsequent section. 

\begin{proposition}[Energy stability]\label{prop:ener_stab}
Under conditions (C1) and (C2) 
the iterates $(u^k)_{k=1,\dots,K}$ of Algorithm~\eqref{alg:simpl}
satisfy for every $1\le L \le K:=\lfloor T/\tau \rfloor$
\[
E_\vphi[u^L] + \tau \sum_{k=1}^L \|d_t u^k\|^2  
+ \frac{\tau^2}{2} \sum_{k=1}^L \int_\O \frac{\vphi'(|\nabla u^{k-1}|)}{|\nabla u^{k-1}|}
|d_t \nabla u^k|^2 \dv{x} \le E_\vphi[u^0].
\]
\end{proposition}

\begin{proof}
Using $v=d_t u^k$ in the equation of Algorithm~\ref{alg:simpl} leads to 
\[
\|d_t u^k\|^2 +
\int_\O \frac{\vphi'(|\nabla u^{k-1}|)}{|\nabla u^{k-1}|} \nabla u^k \cdot d_t \nabla u^k \dv{x}
= 0.
\]
Lemma~\ref{la:orlicz_stab} with $a=\nabla u^{k-1}$ and $b=\nabla u^k$ implies that
\[
\|d_t u^k\|^2 + d_t \int_\O \vphi(|\nabla u^k|) \dv{x} 
+ \frac{1}{2\tau} \int_\O \frac{\vphi'(|\nabla u^{k-1}|)}{|\nabla u^{k-1}|} |\nabla (u^k-u^{k-1})|^2 \dv{x} \le 0,
\]
and summation over $k=1,2,\dots,L$ and multiplication by $\tau$ prove the estimate. 
\end{proof}

\begin{remark}
The stability estimate implies convergence of Richardson-type fixed-point iterations 
for the solution of the stationary $p$-Laplace problem where the step size $\tau$
acts as a damping parameter. For this purpose a stronger
metric to define the evolution such as a weighted $H^1$ product, which mimics 
the $W^{1,p}$ norm may be employed, instead of the $L^2$ inner product,
which in turn acts as a preconditioner for the nonlinear system of
equations. This is an important application of the semi-implicit scheme.
We refer the reader to~\cite{Bart16} for a related approach to a total variation
regularized problem.
\end{remark}

\section{Error estimates}\label{sec:error}
We derive in this section error estimates for the semi-implicit, regularized
numerical scheme of Algorithm~\ref{alg:simpl} with spatial
discretization for the $p$-Dirichlet energy $E_{p,\veps}$. 
We note that all estimates of Section~\ref{sec:stab}
remain valid if spatial discretization is included. 
In what follows we assume that $\cT_h$ is quasi-uniform and that 
$|\cdot|_\veps$ is the standard regularization of euclidean length. 

\subsection{Total variation flow}
We derive an error estimate for the approximation of the gradient 
flow~\eqref{eq:p_flow} with $p=1$ interpreted as a subgradient flow
by the semi-implicit scheme of Algorithm~\ref{alg:simpl}. For this, we compare the iterates 
$(u_h^k)_{k=0,\dots,K} \subset X_h$ in the finite element space 
$X_h = V_h$, i.e., defined via 
\[
(d_t u_h^k,v_h) + \big(|\nabla u_h^{k-1}|_\veps^{-1} \nabla u_h^k, \nabla v_h) = 0,
\]
for all $v_h\in X_h$,  to the iterates $(\tu_h^k)_{k=0,\dots,K} \subset X_h$
of the implicit scheme, i.e., defined via
\[
(d_t \tu_h^k,v_h) + \big(|\nabla \tu_h^k|_\veps^{-1} \nabla \tu_h^k, \nabla v_h) = 0,
\]
for all $v_h\in X_h$. We assume that $\tu_h^0 = u_h^0$. 

\begin{proposition}[Error estimate]\label{prop:diff_impl_simpl_tv}
For the differences of the iterates of the implicit
and the semi-implicit numerical schemes we have that 
\[
\max_{k=0,\dots,K} \|\tu_h^k-u_h^k\| \le c_{1,{\rm s}} \tau^{1/2}  h^{-1} \veps^{-1/2},
\]
where $c_{1,{\rm s}}$ is proportional to $T^{1/2} E_{1,\veps}[u_h^0]$.
\end{proposition}

\begin{proof}
Throughout this proof we omit subscripts~$h$. Taking the difference of the 
numerical schemes we find that $\d^k =\tu^k-u^k$ satisfies 
\[
(d_t \d^k, v) 
+ \Big(\frac{\nabla \tu^k}{|\nabla \tu^k|_\veps} - 
\frac{\nabla u^k}{|\nabla u^{k-1}|_\veps},\nabla v\Big) = 0,
\]
for all $v\in X_h$. Using monotonicity of $a \mapsto a/|a|_\veps$ 
and 1-Lipschitz continuity of $a\mapsto |a|_\veps$, i.e.,
$|d_t|\nabla u^k|_\veps| \le |d_t \nabla u^k|$, 
for $v=\d^k$ we deduce that
\begin{equation}\label{eq:error_eq}
\begin{split}
\frac12 d_t \|\d^k\|^2 + \frac{\tau}{2} \|d_t \d^k\|^2 
&\le - \Big(\frac{\nabla u^k}{|\nabla u^k|_\veps}
- \frac{\nabla u^k}{|\nabla u^{k-1}|_\veps},\nabla \d^k \Big) \\
&= \tau \Big(\frac{\nabla u^k d_t |\nabla u^k|_\veps}
{|\nabla u^k|_\veps|\nabla u^{k-1}|_\veps},\nabla \d^k \Big) \\
&\le \tau \Big(\int_\O \frac{|\nabla d_t u^k|^2}
{|\nabla u^{k-1}|_\veps} \dv{x}\Big)^{1/2}
\Big(\int_\O \frac{|\nabla \d^k|^2}{|\nabla u^{k-1}|_\veps} \dv{x}\Big)^{1/2}.
\end{split}
\end{equation}
Invoking an inverse estimate and $|\nabla u^{k-1}|_\veps \ge \veps$ we infer that
\[
\int_\O \frac{|\nabla \d^k|^2}{|\nabla u^{k-1}|_\veps} \dv{x} 
\le c \veps^{-1} h^{-2} \|\d^k\|^2.
\]
Let $1\le L \le K$ be such that $\|\d^L\| = \max_{k=1,\dots,K} \|\d^k\|$. 
Multiplying~\eqref{eq:error_eq} by $\tau$ and summing over $k=1,2,\dots,L$ shows that  
\[\begin{split}
\|\d^L\|^2 & \le c \tau^2 h^{-1} \veps^{-1/2} \sum_{k=1}^L 
\Big(\int_\O \frac{|\nabla d_t u^k|^2}{|\nabla u^{k-1}|_\veps} 
\dv{x}\Big)^{1/2} \|\d^k\| \\
& \le c \tau^{1/2} h^{-1} \veps^{-1/2}
\Big(\tau^2 \sum_{k=1}^L \int_\O \frac{|\nabla d_t u^k|^2}
{|\nabla u^{k-1}|_\veps} \dv{x}\Big)^{1/2}
\Big(\tau \sum_{k=1}^L \|\d^k\|^2\Big)^{1/2} \\
&\le  c \tau^{1/2} h^{-1} \veps^{-1/2} C_0 (L \tau)^{1/2} \|\d^L\|,
\end{split}\]
where we incorporated the estimate of Proposition~\ref{prop:ener_stab}
with $C_0 = E_{1,\veps}[u^0]$ and $\vphi(r) = |r|_\veps$ so that 
$\vphi'(r)/r = |r|_\veps^{-1}$. Dividing by $\|\d^L\|$ and noting
$L\tau \le T$ implies the asserted estimate. 
\end{proof}

\begin{remark}
In~\cite{FieVee03} a precise characterization of the monotonicity
of the regularized 1-Laplace operator is provided, i.e., we have 
\[
\Big(\frac{a}{|a|_\veps} - \frac{b}{|b|_\veps}\Big) \cdot (a-b) 
= \Big|\frac{(a,\veps)}{|a|_\veps} - \frac{(b,\veps)}{|b|_\veps} \Big|^2
\frac{|a|_\veps + |b|_\veps}{2}.
\]
Unfortunately, we did not succeed in deriving a sharper error estimate making 
use of the identity.
\end{remark}

An error estimate follows from combining Proposition~\ref{prop:diff_impl_simpl_tv} 
with the error estimate~\eqref{eq:p_flow_est} for the implicit scheme
from~\cite{BaNoSa14,BaNoSa15}.

\begin{corollary}\label{thm:tv_flow_simpl}
Let $\O$ be star-shaped and $u^0\in BV(\O)\cap L^\infty(\O)$. Assume that
$\cT_h$ is quasi-uniform and $u_h^0 \in V_h$ is such that $|Du_h^0|(\O) \le c|Du^0|(\O)$. 
If~$u$ solves~\eqref{eq:p_flow} with $p=1$ then we have for the iterates
$(u_h^k)_{k=0,\dots,K}$ of Algorithm~\ref{alg:simpl} with $\vphi(r) = |r|_\veps$ 
and the standard regularization $|\cdot|_\veps$ that 
\[
\max_{k=0,\dots,K} \|u(t_k)-u_h^k\| 
\le c_{{\rm isf}} \tau^{1/2} + 2(c_{1,{\rm r}} T)^{1/2} \veps^{1/2} + c_{1,{\rm i}} h^{1/6}
+ c_{1,{\rm s}} \tau^{1/2}  h^{-1} \veps^{-1/2}.
\]
The factor $\tau^{1/2}$ in the first term can be replaced by $\tau$ if
$\p E_{1,0}[u^0] \neq \emptyset$. The factor $h^{1/6}$ in the
third term can be replaced by $h^{1/4}$ for special uniform cartesian
meshes and definition of total variation using the $\ell^1$-norm in $\mathbb{R}^d$.
\end{corollary}

\subsection{$p$-Laplace gradient flow}
In case $p>1$ a stronger estimate follows from the
strong monotonicity of the $p$-Laplace operator. We argue as
in the previous subsection and compare the finite element
iterates of the semi-implicit scheme defined via
\[
(d_t u_h^k, v_h) + \Big(\frac{\vphi'(|\nabla u_h^{k-1}|)}{|\nabla u_h^{k-1}|} \nabla u_h^k, \nabla v_h\Big) = 0,
\]
for all $v_h\in X_h$, to those of the implicit scheme 
\[
(d_t \tu^k,v_h) + \Big(\frac{\vphi'(|\nabla \tu_h^k|)}{|\nabla \tu_h^k|} \nabla \tu_h^k, \nabla v_h\Big) = 0,
\]
for all $v_h\in X_h$, where we assume that $u_h^0 = \tu_h^0$. To 
simplify our calculations, we define the operator $A:\R^d\to \R^d$ via 
\[
A(a) = \frac{\vphi'(|a|)}{|a|} a, 
\]
and the function $\vphi_\a:[0,\infty)\to [0,\infty)$ given for
$\a,s\ge 0$ by $\vphi_\a(0)=0$ and 
\[
\vphi_\a'(s) = \frac{\vphi'(\a+ s)}{\a + s} s.
\]
We also use the notation $a\lesssim b$ if there exists a constant $c>0$ such
that $a\le cb$; we write $a\eqsim b$ if $a\lesssim b$ and $b\lesssim a$. 
We assume further properties of $\vphi$.

\medskip

{\bf Condition} (C3){\bf .} The function $\vphi \in C(\R_{\ge 0}) \cap  C^2(\R_{>0})$ is 
convex and positive on $(0,\infty)$,
satisfies $\vphi(0)=0$, and $\lim_{s\to 0} \vphi(s)/s = 0$ and 
$\lim_{s\to \infty} \vphi(s)/s=\infty$; moreover $\vphi$ and its convex
conjugate $\vphi^*$ satisfy $\vphi(2s)\lesssim \vphi(s)$ and
$\vphi^*(2r)\lesssim \vphi(r)$ 
for all $r,s \in \R_{\ge 0}$; additionally we have $\vphi''(s) s \eqsim \vphi'(s)$. 

\medskip

The functions defined in Example~\ref{ex:orlicz_fns} satisfy~(C3) for $p>1$ 
with constants that deteriorate as $p\searrow 1$. 

\begin{lemma}\label{la:mon_props_b}
If $\vphi$ satisfies~(C3), then the following statements are valid. \\
(i) For all $a,b\in \R^d$ we have 
\begin{align}
\big(A(a)-A(b)\big)\cdot (a-b) & \eqsim  \vphi_{|a|} (|a-b|), \label{eq:op_A_mon} \\
\big|A(a)-A(b)\big| &\lesssim \vphi_{|a|}'(|a-b|), \label{eq:op_A_cont}
\end{align}
and
\begin{equation}\label{eq:sim_phi}
\vphi_{|a|} (|a-b|) \eqsim \frac{\vphi'(|a|+|b|)}{|a|+|b|} |a-b|^2.
\end{equation}
(ii) For all $\a,r,s \ge 0$ and $\d>0$ we have
\begin{equation}\label{eq:young_phi}
\vphi_\a'(r) s \le c_\d \vphi_\a(r) + \d \vphi_\a(s).
\end{equation}
\end{lemma}

\begin{proof}
We refer the reader to~\cite{DieEtt08} for proofs of the estimates. 
\end{proof}

The relations of Lemma~\ref{la:mon_props_b} lead to the following result. 

\begin{proposition}[Error estimate]\label{prop:diff_impl_simpl_p}
Suppose that $\vphi$ satisfies (C1)-(C3) and that there exist constants 
$c_1,c_2>0$ such that 
\begin{equation}\label{vphi'(s)/s}
c_1 \max\{s,\veps\}^{p-2} \le \frac{\vphi'(s)}{s} \le c_2 \, \veps^{p-2}
\end{equation}
for all $s\ge 0$. Assume further that $\cT_h$ is quasi-uniform and
there exists $c_\infty>0$ such that 
\begin{equation}\label{max-norm-bound}
\max_{k=0,\dots,K} \|u_h^k\|_{L^\infty(\O)} 
+  \max_{k=0,\dots,K} \|\tu_h^k\|_{L^\infty(\O)} \le c_\infty.
\end{equation}
Then, for the differences of the iterates of the implicit
and the semi-implicit numerical schemes we have that 
\[
\max_{k=0,\dots,K} \|\tu_h^k-u_h^k\| \le c_{p,{\rm s}} \tau^{1/2}  (h\veps)^{(p-2)/2},
\]
where $c_{p,{\rm s}}$ is proportional to $\big(E_{\vphi}[u_h^0]\big)^{1/2}$.
\end{proposition}

\begin{proof}
We omit the subscripts $h$ in what follows. To derive an estimate 
for $\d^k = u^k - \tu^k$ we test the difference of the equations that define $u^k$ and $\tu^k$
with $\d^k$ and use~\eqref{eq:prod} in conjunction with
\eqref{eq:op_A_mon} to verify that 
\[\begin{split}
\frac{d_t}{2} & \|\d^k\|^2 + \frac{\tau}{2} \|d_t \d^k\|^2 + \int_\O \vphi_{|\nabla u^k|} (|\nabla \d^k|) \dv{x} \\
&\lesssim (d_t \d^k,\d^k) + \big(A(\nabla u^k)-A(\nabla \tu^k),\nabla\d^k\big) \\
&= \big(A(\nabla u^k) -A(\nabla u^{k-1}), \nabla \d^k\big) 
+\Big(\frac{\vphi'(|\nabla u^{k-1}|)}{|\nabla u^{k-1}|} \nabla [u^{k-1}-u^k],\nabla \d^k\Big) \\&  = R_1 + R_2. 
\end{split}\]
To bound $R_1$ we use that~\eqref{eq:op_A_cont} 
and~\eqref{eq:young_phi} imply that
\[\begin{split}
R_1 &\eqsim \int_\O \vphi_{|\nabla u^k|}'(|\nabla [u^k-u^{k-1}]|) |\nabla \d^k| \dv{x} \\
& \le \d \int_\O \vphi_{|\nabla u^k|} (|\nabla \d^k|)\dv{x}  
+ c_\d \int_\O \vphi_{|\nabla u^k|} (|\nabla [u^k-u^{k-1}]|)\dv{x}.
\end{split}\]
Invoking the equivalence~\eqref{eq:sim_phi}, the property
that $s\mapsto \vphi'(s)/s$ is nonincreasing, and the relation $\tau d_t u^k = u^k-u^{k-1}$, 
we deduce that
\[\begin{split}
R_1 &\lesssim \d \int_\O \vphi_{|\nabla u^k|} (|\nabla \d^k|)\dv{x}  
+ c_\d \tau^2 \int_\O \frac{\vphi'(|\nabla u^k| + |\nabla u^{k-1}|)}{|\nabla u^k| + |\nabla u^{k-1}|}
|d_t \nabla u^k|^2 \dv{x}  \\
&\le \d \int_\O \vphi_{|\nabla u^k|} (|\nabla \d^k|)\dv{x}
+  c_\d \tau^2 \int_\O  \frac{\vphi'(|\nabla u^{k-1}|)}{|\nabla u^{k-1}|} |d_t \nabla u^k|^2 \dv{x}.
\end{split}\]
For the term $R_2$ we employ Young's inequality
$st\le \delta s^2 + c_\delta t^2$ to obtain
\[\begin{split}
&R_2 \le \d \int_\O \frac{\vphi'(|\nabla u^k|+|\nabla \tu^k|)}{|\nabla u^k|+|\nabla \tu^k|}
|\nabla \d^k|^2 \dv{x} \\
& + c_\d \tau^2 \Big\|\frac{|\nabla u^k|+|\nabla \tu^k|}{\vphi'(|\nabla u^k|+|\nabla \tu^k|)} 
\frac{\vphi'(|\nabla u^{k-1}|)}{|\nabla u^{k-1}|} \Big\|_{L^\infty(\O)} 
 \int_\O \frac{\vphi'(|\nabla u^{k-1}|)}{|\nabla u^{k-1}|} |d_t \nabla u^k|^2 \dv{x}.
\end{split}\]
Utilizing an inverse estimate in conjunction with
\eqref{max-norm-bound} yields
\[
\|\nabla u^k\|_{L^\infty(\Omega)} +
\|\nabla\tu^k\|_{L^\infty(\Omega)} \lesssim c_\infty h^{-1},
\]
whence \eqref{vphi'(s)/s} gives
\[
\Big\|\frac{|\nabla u^k|+|\nabla \tu^k|}{\vphi'(|\nabla u^k|+|\nabla
  \tu^k|)}\Big\|_{L^\infty(\Omega)} \lesssim c_\infty h^{p-2},
\quad
\Big\|\frac{\vphi'(|\nabla u^{k-1}|)}{|\nabla u^{k-1}|}
\Big\|_{L^\infty(\O)} \lesssim \veps^{p-2}.
\]
In view of~\eqref{eq:sim_phi}, these bounds lead to
\[\begin{split}
R_2 \le \d \int_\O \vphi_{|\nabla u^k|}(|\nabla \d^k|) \dv{x} 
+ c c_\d \tau^2 h^{p-2} \veps^{p-2}  
\int_\O \frac{\vphi'(|\nabla u^{k-1}|)}{|\nabla u^{k-1}|} |d_t \nabla u^k|^2 \dv{x}.
\end{split}\]
Combining the first estimate with those of $R_1$ and $R_2$ we 
obtain the following bound
after summation over $k=1,2,\dots,L$ and multiplication by $\tau$ 
\[\begin{split}
\|\d^L\|^2 + & \tau^2 \sum_{k=1}^L \|d_t \d^k\|^2 
+  \tau \sum_{k=1}^L \int_\O \vphi_{|\nabla u^k|} (|\nabla \d^k|) \dv{x} \\
& \lesssim \tau \big(1 + (h\veps)^{p-2}\big) \tau^2  \sum_{k=1}^L 
\int_\O \frac{\vphi'(|\nabla u^{k-1}|)}{|\nabla u^{k-1}|} |d_t \nabla u^k|^2 \dv{x},
\end{split}\]
where we have also used that $\d^0=0$. The bound of Proposition~\ref{prop:ener_stab} 
for the sum on the right-hand side implies the asserted estimate. 
\end{proof}

A complete error estimate follows from combining Proposition~\ref{prop:diff_impl_simpl_p}
with the error estimate~\eqref{eq:p_flow_est} for the implicit scheme
from~\cite{DiEbRu07}.

\begin{corollary}\label{thm:p_flow_simpl}
Let $\O$ be convex, $X=W^{1,p}_0(\O)$, and let $u^0\in W^{1,2}_0(\O)$ and 
$\diver \big(|\nabla u^0|^{p-2}\nabla u^0\big) \in L^2(\O)$.
Let $\cT_h$ be quasi-uniform, $u_h^0 \in X_h$ be such that 
$\|\nabla u_h^0\|_{L^p(\O)} \le c \|\nabla u^0\|_{L^p(\O)}$, and 
$c_\infty>0$ satisfy 
\[
\max_{k=0,\dots,K} \|u_h^k\|_{L^\infty(\O)} 
+  \max_{k=0,\dots,K} \|\tu_h^k\|_{L^\infty(\O)} \le c_\infty.
\]
If $u$ is the solution of~\eqref{eq:p_flow} with $p\in (1,2)$ and
$(u_h^k)_{k=0,\dots,K}$ are the iterates of Algorithm~\ref{alg:simpl} with 
$\vphi(r) = (|r|_\veps^p-|0|_\veps^p)/p$, then we have
\[
\max_{k=0,\dots,K} \|u(t_k) - u_h^k\| \le c_{{\rm isf}} \tau + c_{p,{\rm i}} h 
+ 2(c_{p,{\rm r}}T)^{1/2} \veps^{p/2} + c_{p,{\rm s}} \tau^{1/2} (h\veps)^{(p-2)/2}.
\]
\end{corollary}

Establishing rigorously the $L^\infty$ bounds \eqref{max-norm-bound}
requires further conditions. Such
bounds can be avoided if in the proof of Proposition~\ref{prop:diff_impl_simpl_p}
inverse estimates $\|\nabla v_h\|_{L^\infty(\O)} \le c h^{-d/p} \|\nabla v_h\|_{L^p(\O)}$ 
are used, which leads to a weaker error estimate since $d/p>1$.  

\begin{remark}
The $L^\infty$ bounds \eqref{max-norm-bound}
can be obtained via discrete maximum principles
provided that $u^0\in L^\infty(\O)$. 
For the semi-implicit scheme it is sufficient to guarantee that the system matrix 
in every time step is an $M$-matrix, which holds if quadrature (mass lumping) 
is used, the triangulation is (strongly) acute, and $\tau$ is sufficiently small. 
For the implicit scheme this follows from monotonicity properties
of the minimization problems at each time step, which are available if
quadrature is used and the mesh is acute. 
\end{remark}

\section{Numerical experiments}\label{sec:numex}
We illustrate our theoretical findings by numerical experiments for the
most singular case $p=1$. For this, we construct explicit solutions and
then compare errors for approximations obtained with the implicit 
scheme and the semi-implicit scheme of Algorithm~\ref{alg:simpl} 
and different regularization parameters. 
The nonlinear systems of equations in the time steps of the
implicit scheme were solved with an alternating direction
method of multipliers (ADMM) with variable step sizes proposed and analyzed
in~\cite{BarMil17-pre}. 

\subsection{Explicit solutions}
We consider~\eqref{eq:p_flow} with $p=1$ and Dirichlet boundary conditions, i.e., 
formally, we consider 
\begin{equation}\label{eq:tv_flow_db}
\p_t u = \diver \frac{\nabla u}{|\nabla u|}, \quad 
u(0,\cdot) = u^0, \quad u(t,\cdot)|_\pO = 0.
\end{equation}
Establishing
the existence of solutions subject to Dirichlet boundary 
conditions is a difficult task but the stability and error estimates
remain valid whenever a solution exists. To
construct explicit, nontrivial solutions we use the equivalent
formulation 
\begin{equation}\label{eq:tv_flow_mixed_a}
u_t = \diver p, \quad \nabla u  \in \p I_K(p),
\end{equation}
where $K=\overline{B_1(0)}$. 
The inclusion follows from its equivalence to $p \in \p |\nabla u|$
and means that $p\in L^\infty(\O;\R^d)$ with $|p|\le 1$ satisfies 
\[
(\nabla u, q-p)  \le 0
\]
for all $q\in L^\infty(\O;\R^d)$ with $|q|\le 1$, provided that 
$\nabla u\in L^1(\O;\R^d)$.
For the case that $u\in BV(\O)\cap L^2(\O)$ with $u|_\pO=0$ we may formulate
it as 
\begin{equation}\label{eq:tv_flow_mixed_b}
-\big(u,\diver(q-p)\big) \le 0,
\end{equation}
requiring that $p,q\in H(\diver;\O)$ with $|p|,|q|\le 1$. 
We refer the reader to~\cite{BeCaNo02,BaNoSa14} for further details. 
The following examples use that for regular solutions of~\eqref{eq:tv_flow_db} 
the change of height $\p_t u$ at a noncritical point $x\in \O$  equals 
the negative mean curvature~$-H = \diver(\nabla u/|\nabla u|)$ of the
corresponding level set,
and that jump sets, along which gradients are unbounded, have vanishing
normal velocity $V = \p_t u/|\nabla u| = -H/|\nabla u|$.

\begin{example}[Decreasing disk,~\cite{BeCaNo02}]\label{ex:decr_disk} 
Let $\O\subset \R^d$ such that $B_1(0)\subset \O$ and 
\[
u(t,x) = \max\big\{1-td,0\big\} \chi_{B_1(0)}(x).
\]
Then $u$ solves~\eqref{eq:tv_flow_db} with $u^0 = \chi_{B_1(0)}$. 
\end{example}

\begin{proof}
For $t\le 1/d$ and $x\in \O$ we define
\[
p(t,x) = - 
\begin{cases}
x, & |x|\le 1, \\
x/|x|^d, & |x| \ge 1.
\end{cases}
\]
For $t>1/d$ we set $p(t,x)=0$. We have that $p(t,\cdot)$ is 
continuous in $\O$ with $|p|\le 1$ and $\p_t u = \diver p$ in~$\O$.
To show that $u$ solves~\eqref{eq:tv_flow_mixed_a} it remains to 
verify~\eqref{eq:tv_flow_mixed_b}. For $q\in H(\diver;\O)$ with
$|q|\le 1$ we have 
\[
-\big(u,\diver(q-p)\big) 
= -(1-td) \int_{\p B_1(0)} (q-p)\cdot n \dv{s} \le 0,
\]
since $p=-n$ on $\p B_1(0)$ and $q \cdot n \le 1$. 
\end{proof}

The solution constructed in the second example is Lipschitz
continuous at all times but the discontinuity set of $\nabla u$ 
is nonstationary. Moreover, we have that $\p_t u(0) \not \in L^2$
so that only the suboptimal convergence rate $\cO(\tau^{1/2})$
for the time-discretization error can be expected. 

\begin{example}[Decreasing cone]\label{ex:decr_cone}
Let $\O \subset \R^d$ such that $B_1(0)\subset \O$ and 
\[
u^0(x) =  \max\big\{1-|x|,0\}.
\]
If
\[
s(t) = (d+1)^{1/2} t^{1/2}, \quad 
r(t) = \frac12 \big(1 + (1-4t(d-1)\big)^{1/2}\big),
\]
then for $t\le (d+1)/(4d^2)$ we have 
\[
u(t,x) = 
\begin{cases}
1-s(t)-t(d-1)/s(t), & |x|\le s(t), \\
1-|x|- t(d-1)/|x|, & s(t)\le |x| \le r(t), \\
0, & r(t) \le |x|.
\end{cases}
\]
For $t \ge (d+1)/(4d^2)$, we have $u(t,x)= 0$ for all $x\in \O$. 
\end{example}

\begin{proof}
We first note that for a nondegenerate point $x\in \O$ for
a solution of~\eqref{eq:tv_flow_db} we have
that the mean curvature of its level set equals $(d-1)/|x|$, whence
\[
\p_t u(t,x) = -\frac{d-1}{|x|},
\]
as long as $\nabla u(t,x)\neq 0$; hence, $u(t,x) = 1-|x|- t(d-1)/|x|$. 
To prove that $u$ is a solution of \eqref{eq:tv_flow_mixed_a},
we construct an appropriate vector field $p$. We define
\[
p(t,x) = -
\begin{cases}
x/s(t), & |x| \le s(t), \\
x/|x|, & s(t) \le |x| \le r(t), \\
xr(t)^{d-1}/|x|^d, & r(t)\le |x|,
\end{cases}
\]
and note that $p(t,\cdot)$ is continuous in $\O$ with $|p|\le 1$ and 
\[
\diver p(t,x) = -
\begin{cases}
d/s(t), & |x|< s(t), \\
(d-1)/|x|, & s(t) < |x| < r(t), \\
0, & |x| > r(t) .
\end{cases}
\]
The differential equation  $\p_t u = \diver p$ is obviously 
satisfied for $|x| > s(t)$. For $0\le |x| < s(t)$ we obtain the condition
\[
-s' - \frac{d-1}{s} + \frac{t(d-1)}{s^2}s' = -\frac{d}{s}  \quad \Longleftrightarrow \quad
s' = \frac{s}{s^2-(d-1)t},
\]
which is satisfied by definition of $s$. We finally note 
that, since $u(t,\cdot) \in W^{1,\infty}(\O)$ with $u(t,\cdot)|_{\pO}=0$
and $p= \nabla u/|\nabla u|$ for $s(t)\le |x| \le r(t)$
and $\nabla u = 0$ otherwise, we have 
\[\begin{split}
-\big(u,\diver(q-p)\big) = \int_\O \nabla u \cdot (q-p)\dv{x} 
\le 0,
\end{split}\]
provided that $|q|\le 1$. This proves the statement.
\end{proof}

Snapshots of implicit approximations of the total variation flow
with $\veps=0$ on a triangulation~$\cT_\ell$ of
$\O=(-3/2,3/2)^2$ obtained from $\ell=5$ uniform refinements of an initial 
partitions $\cT_0$ into two triangles and with $\tau = h/4$ are shown 
in Figures~\ref{fig:decr_disk} and~\ref{fig:decr_cone}.


\begin{figure}[p]
\includegraphics[width=.45\linewidth]{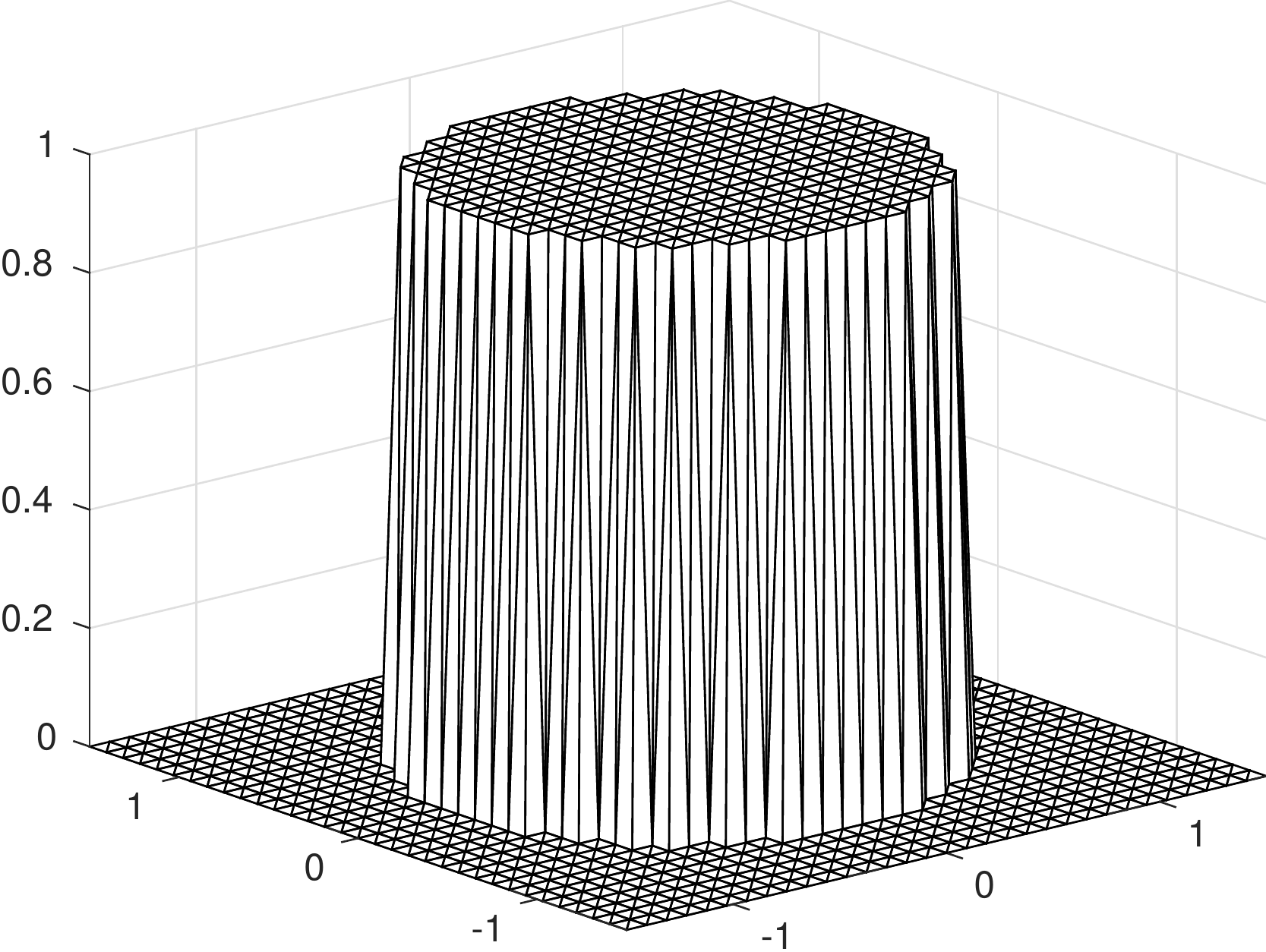} \hspace*{2mm}
\includegraphics[width=.45\linewidth]{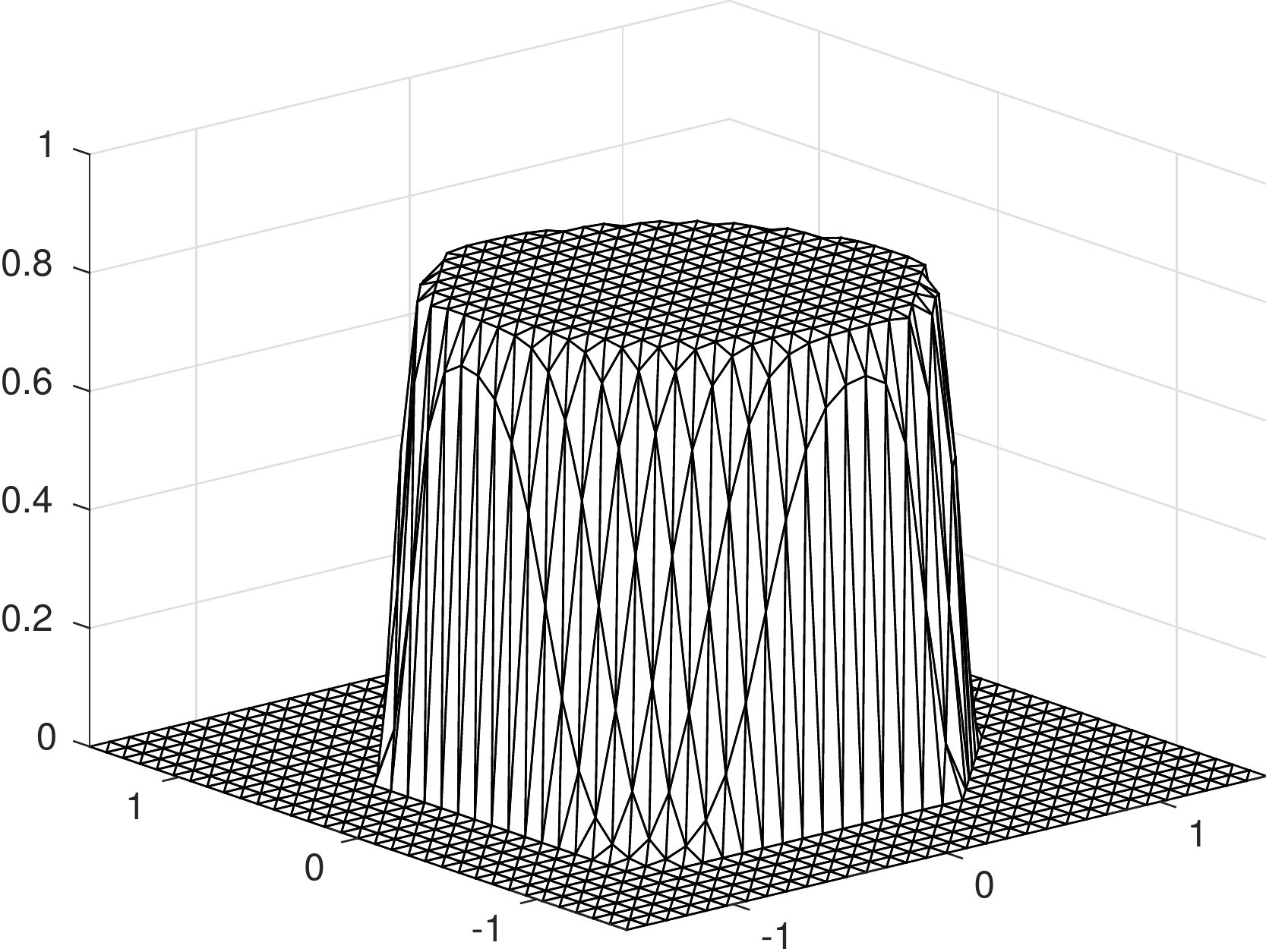} \\
\includegraphics[width=.45\linewidth]{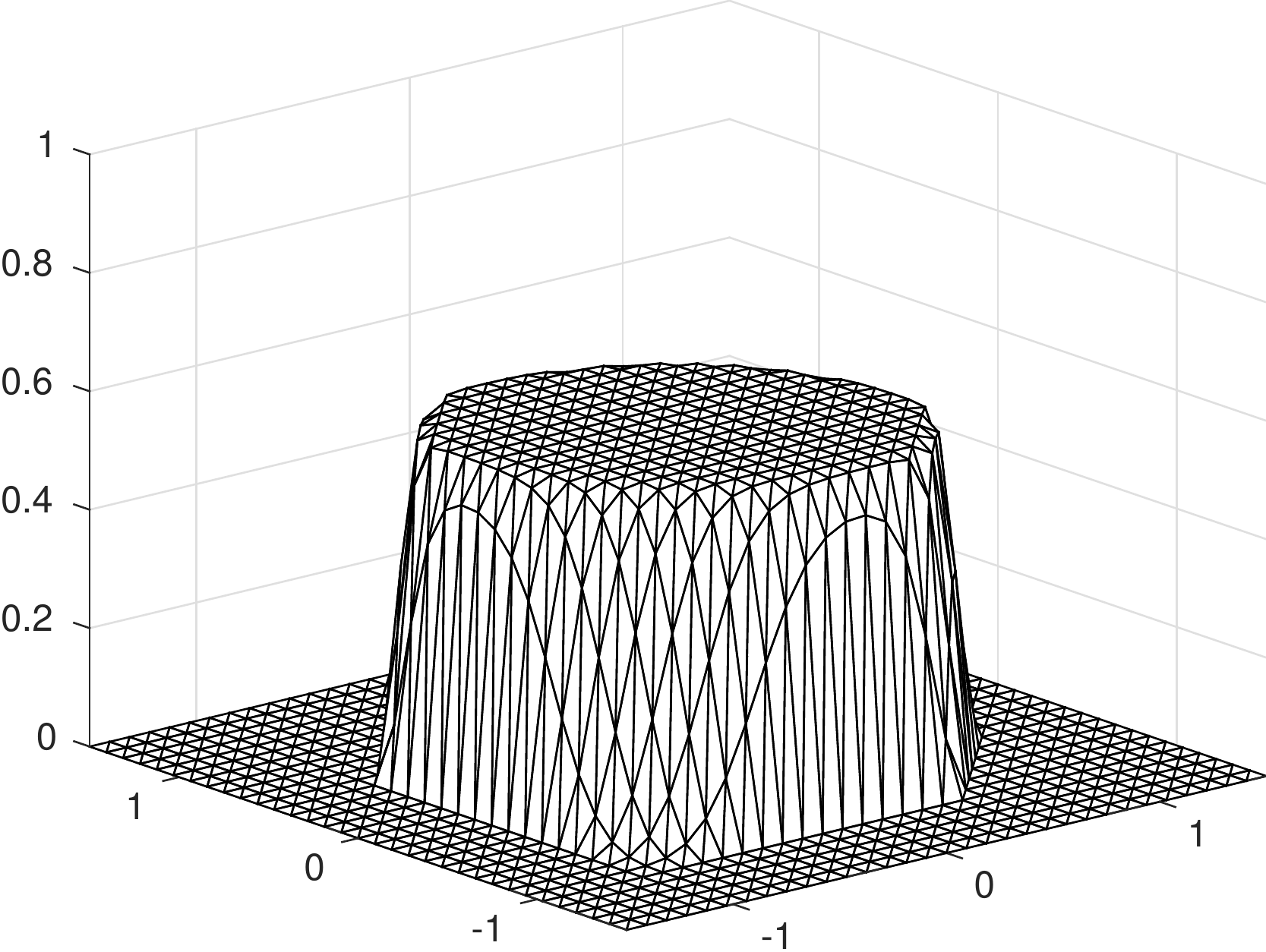} \hspace*{2mm}
\includegraphics[width=.45\linewidth]{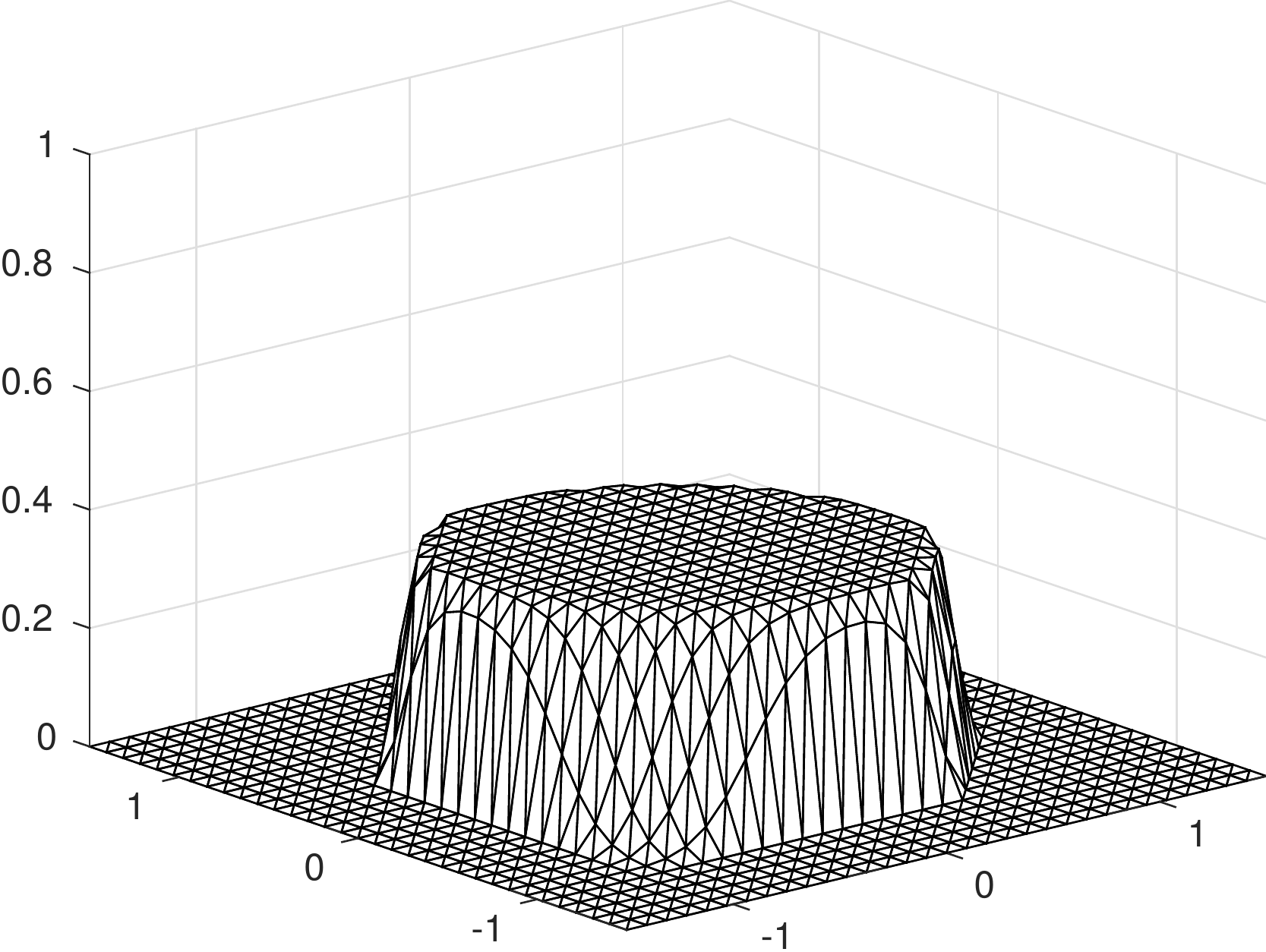} 
\caption{\label{fig:decr_disk} Numerical solutions for $t\approx 0.0, 0.1, 0.2, 0.3$
in Example~\ref{ex:decr_disk} computed with the implicit scheme and $\veps=0$.}
\end{figure}

\begin{figure}[p]
\includegraphics[width=.45\linewidth]{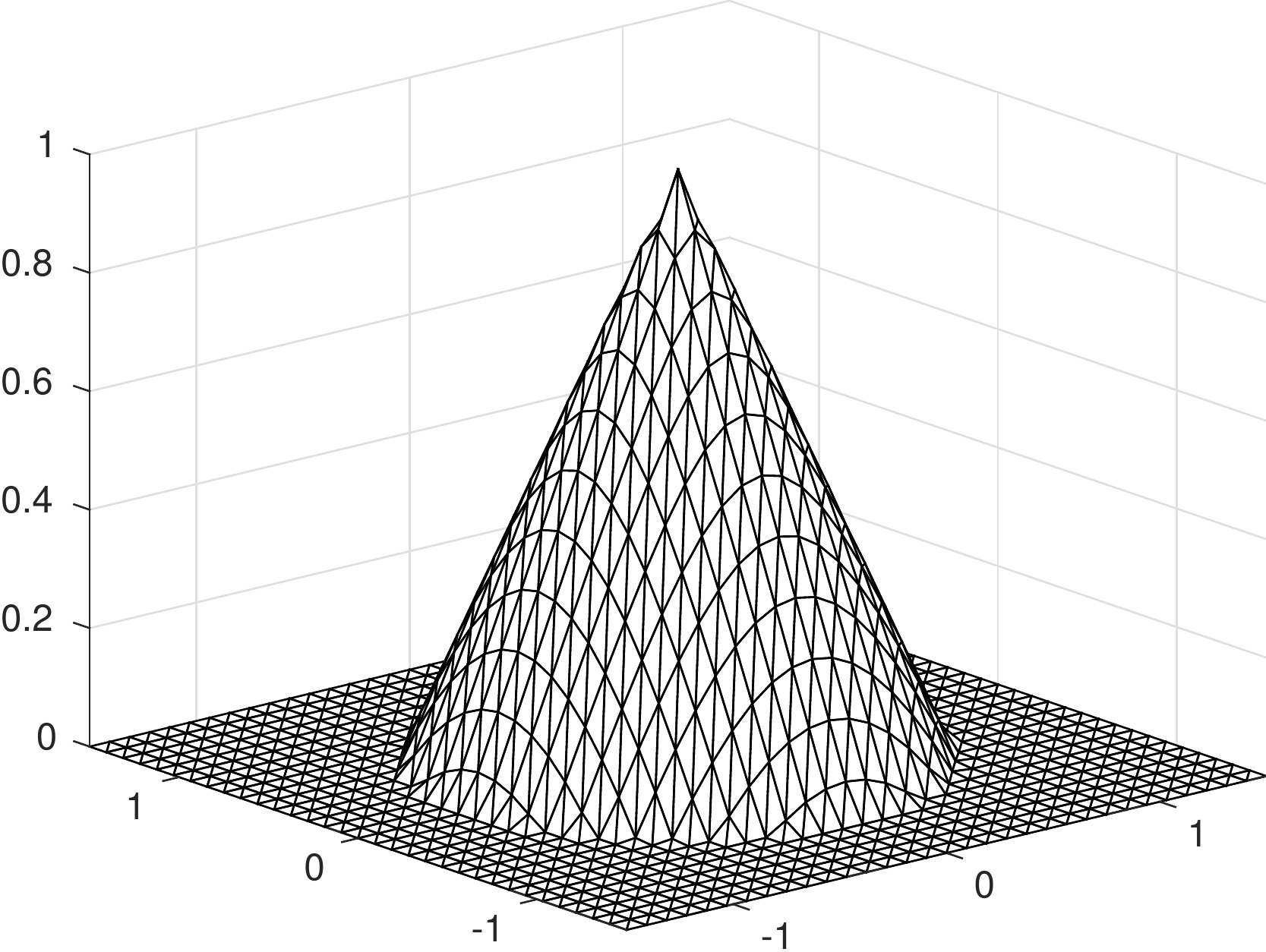} \hspace*{2mm}
\includegraphics[width=.45\linewidth]{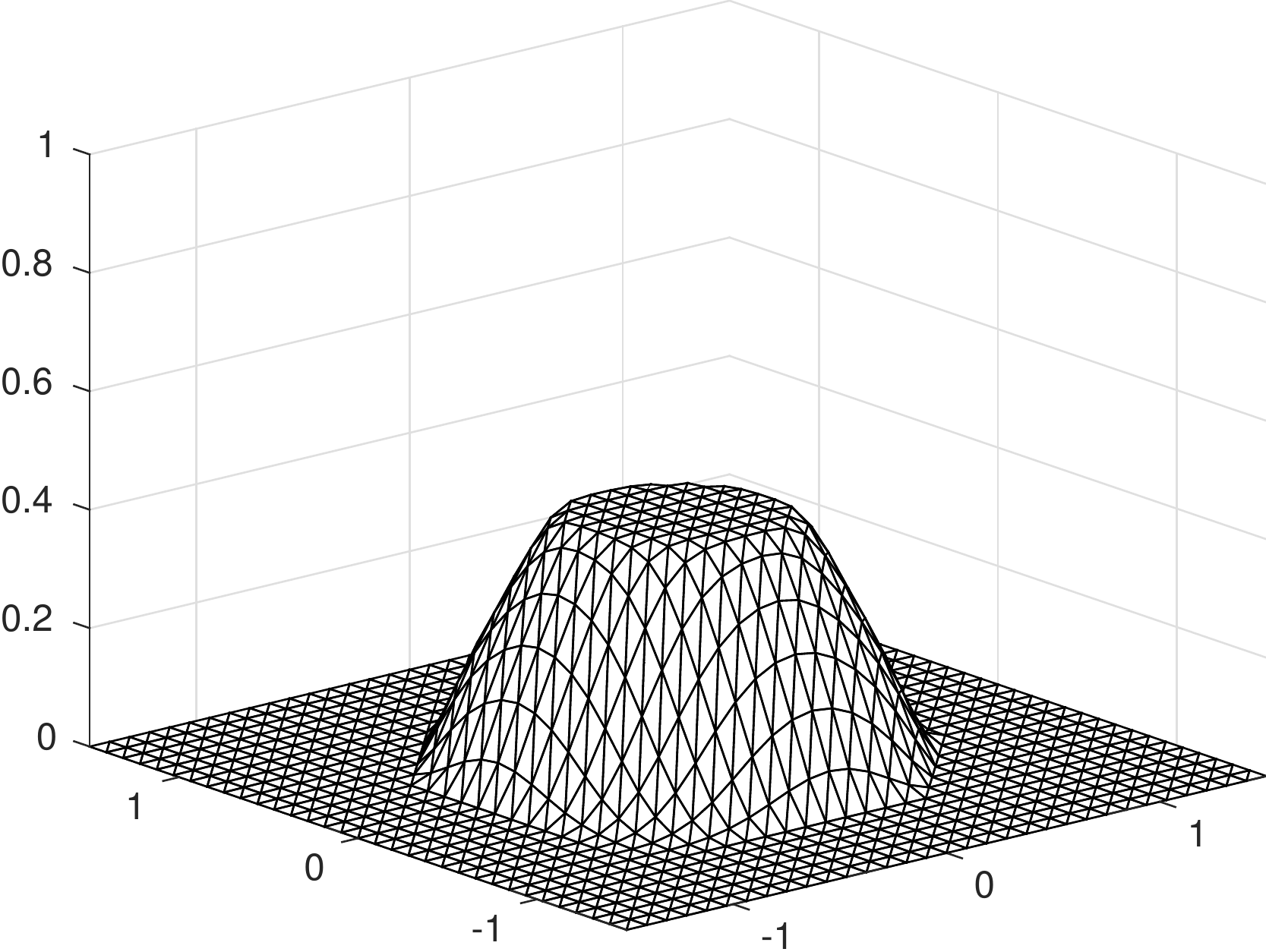} \\
\includegraphics[width=.45\linewidth]{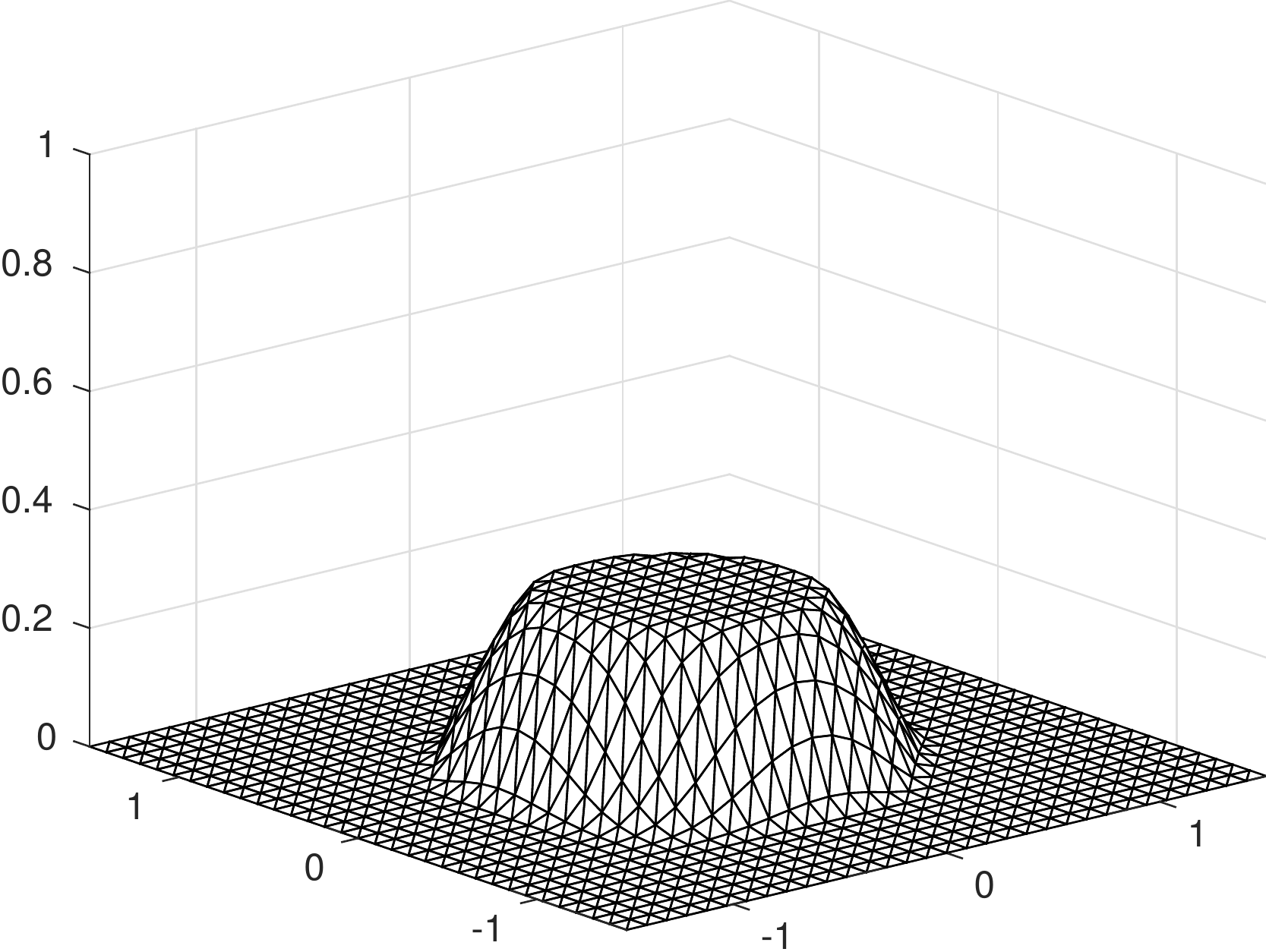} \hspace*{2mm}
\includegraphics[width=.45\linewidth]{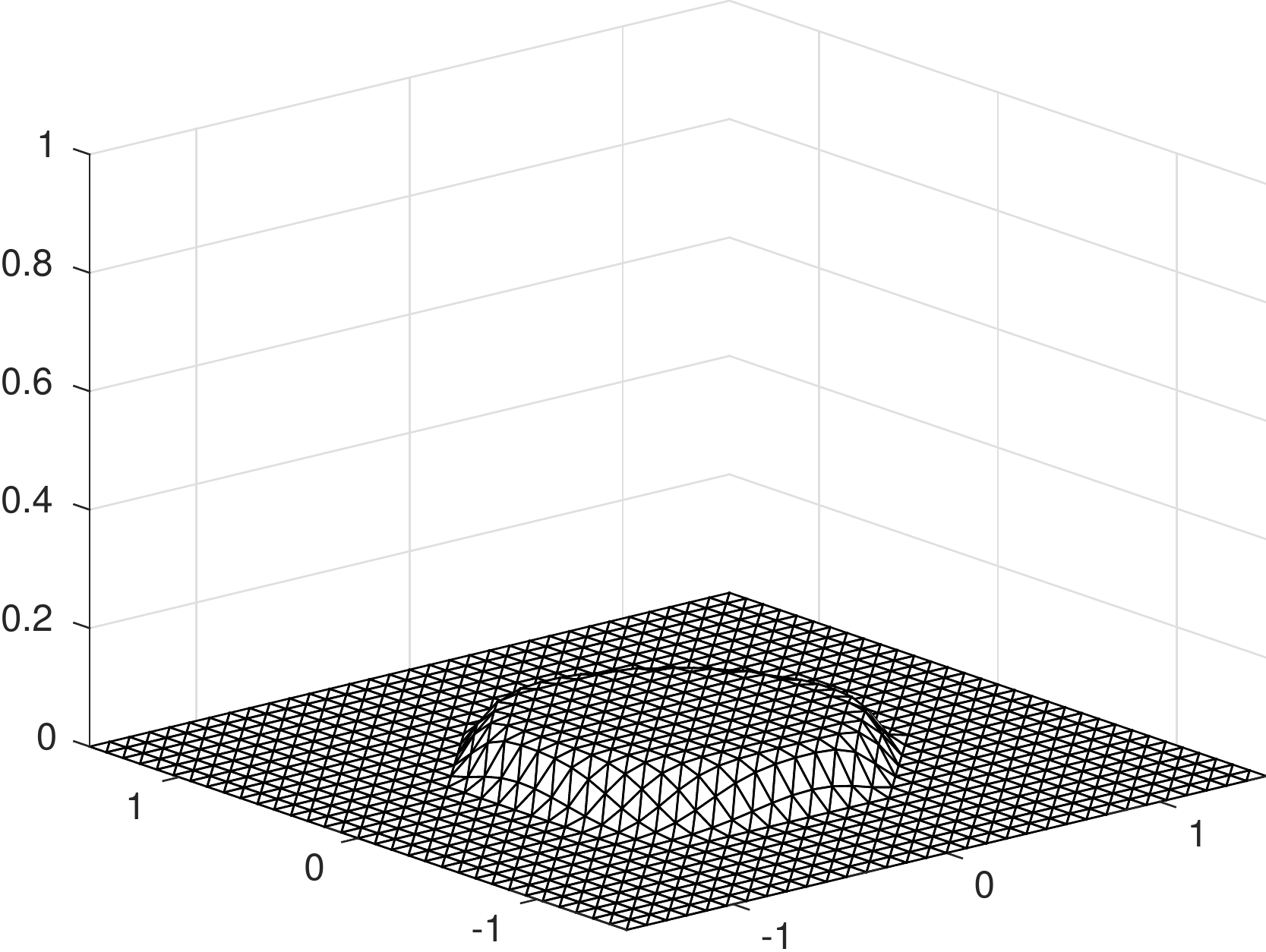} 
\caption{\label{fig:decr_cone} Numerical solutions for $t \approx 0.0, 0.05, 0.1, 0.15$
in Example~\ref{ex:decr_cone} computed with the implicit scheme and $\veps=0$.}
\end{figure}



\begin{figure}[p]
\includegraphics[width=.48\linewidth]{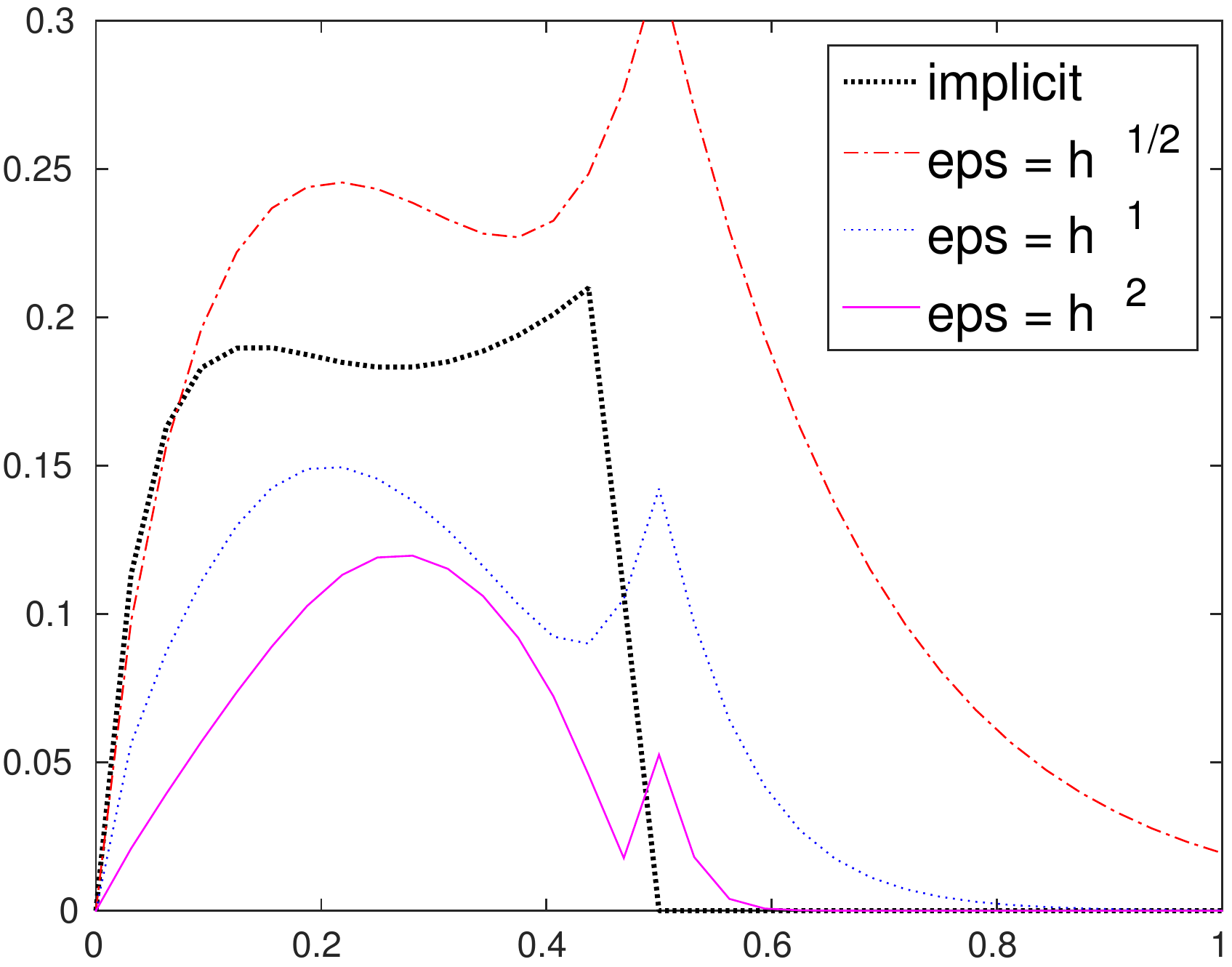} \hspace*{2mm}
\includegraphics[width=.48\linewidth]{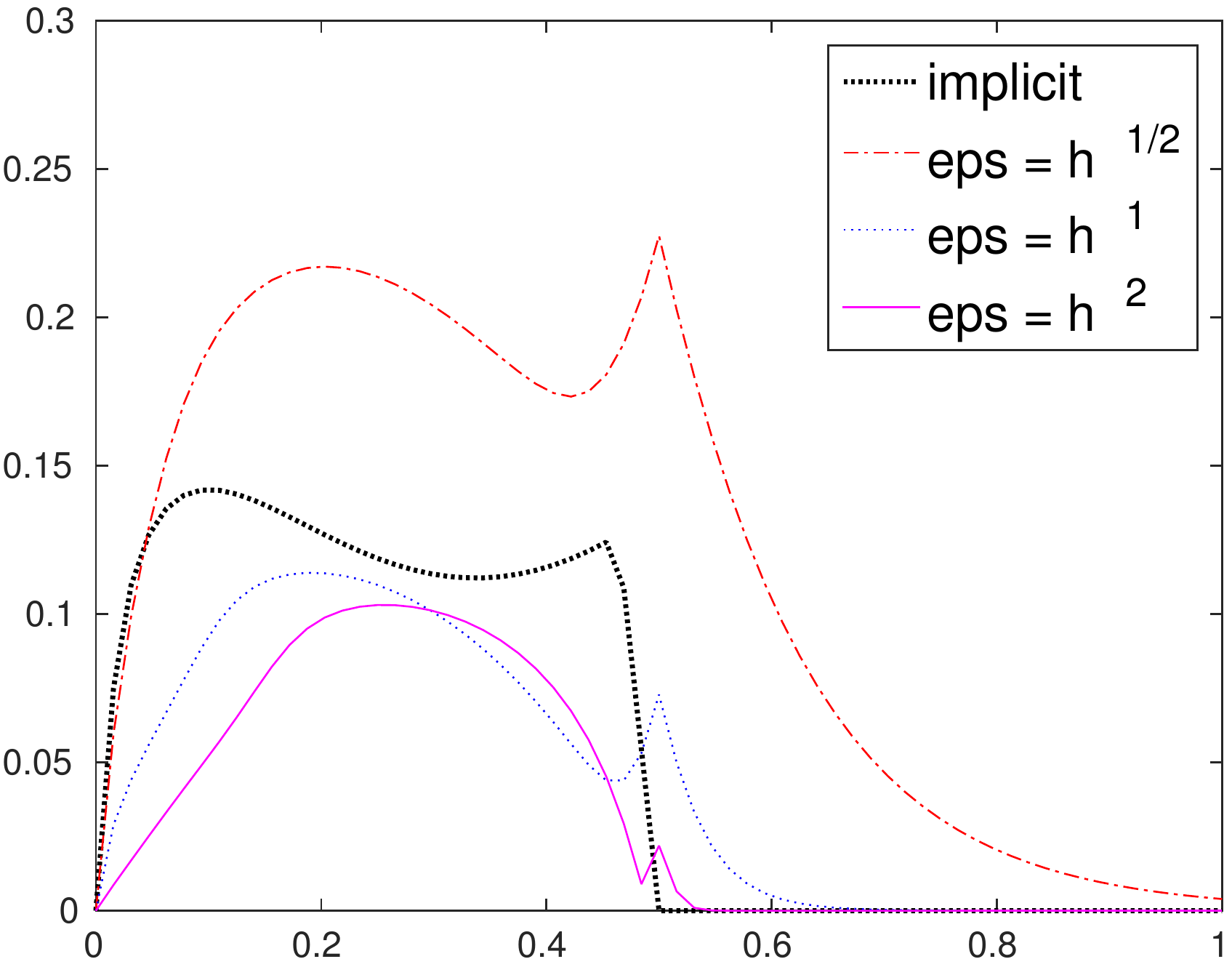} 
\caption{\label{fig:ex_1_l2_errs_time} $L^2$ errors as functions of
$t\in [0,1]$ in Example~\ref{ex:decr_disk} for the semi-implicit scheme with
$\veps = h^\a$, $\a=1/2,1,2$, and implicit approximations on triangulations
$\cT_\ell$, $\ell=4$ (left) and $\ell=5$ (right).}
\end{figure}

\begin{figure}[p]
\begin{minipage}{.55\linewidth}
{\small \begin{tabular}{|c|c|c|c|c|} \hline
$\ell$ & implicit & ${\veps = h^{1/2}}$ & ${\veps = h}$ & ${\veps = h^2}$ \\\hline\hline  
3 & 0.3135 & 0.4024 & 0.2515 & 0.1342 \\\hline
4 & 0.1999 & 0.3179 & 0.1495 & 0.1197 \\\hline
5 & 0.1421 & 0.2276 & 0.1139 & 0.1030 \\\hline
6 & 0.1313 & 0.1882 & 0.1005 & 0.0980 \\\hline
7 & --     & 0.1487 & 0.0813 & 0.0786 \\\hline
8 & --     & 0.1172 & 0.0701 & 0.0679 \\\hline
9 & --     & 0.0908 & 0.0595 & 0.0576 \\\hline
10 & --    & 0.0710 & 0.0510 & 0.0496 \\\hline
\end{tabular}}
\end{minipage}
\begin{minipage}{.44\linewidth}
\includegraphics[width=\linewidth]{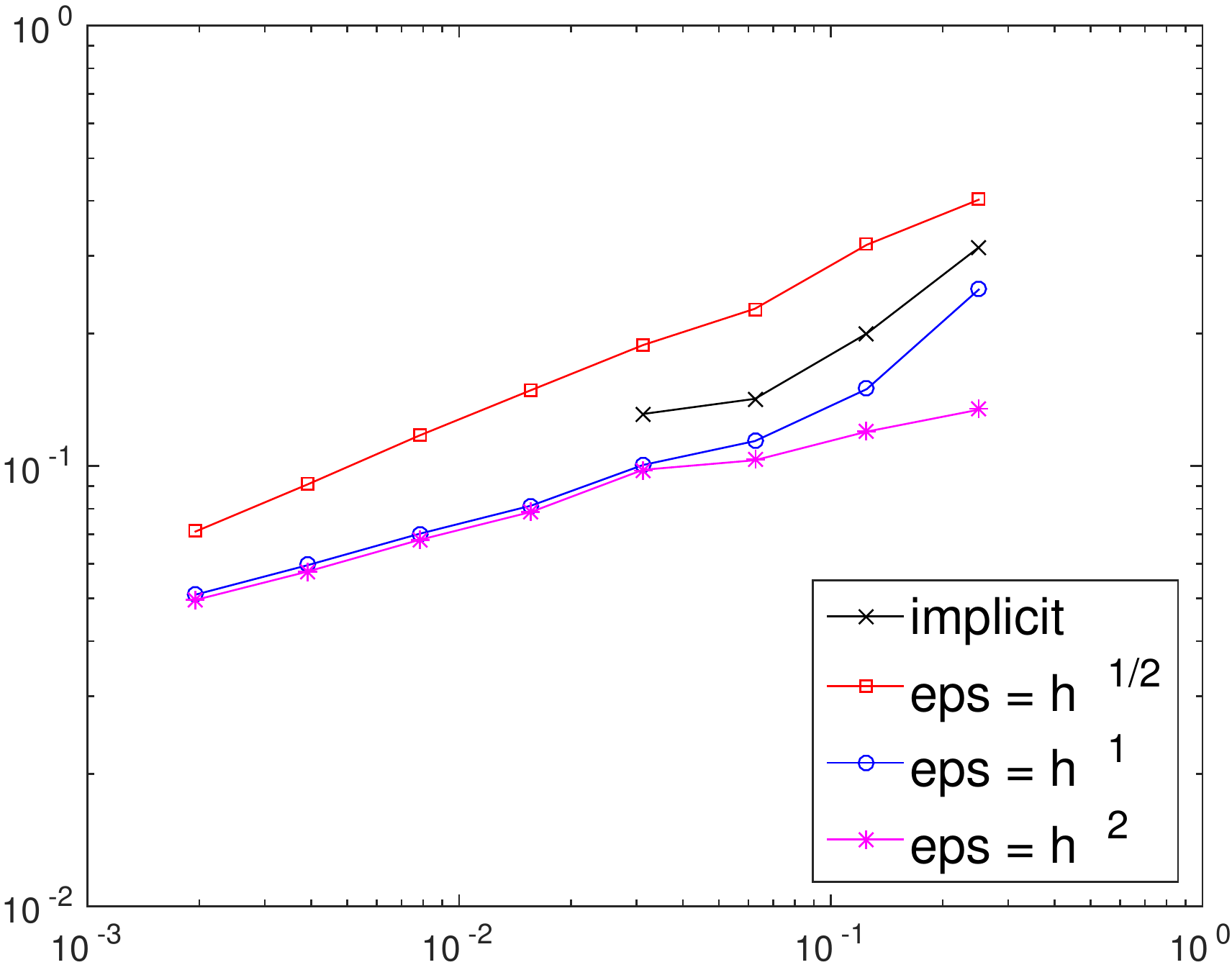}
\end{minipage}
\vspace*{-2mm}
\caption{\label{tab:ex_1_l2_errs} Maximal $L^2$ errors for different choices of 
$\veps$ and on different triangulations~$\cT_\ell$ of level $\ell$
in Example~\ref{ex:decr_disk}.}
\vspace*{2mm}
\end{figure}

\begin{figure}[p]
\includegraphics[width=.3\linewidth]{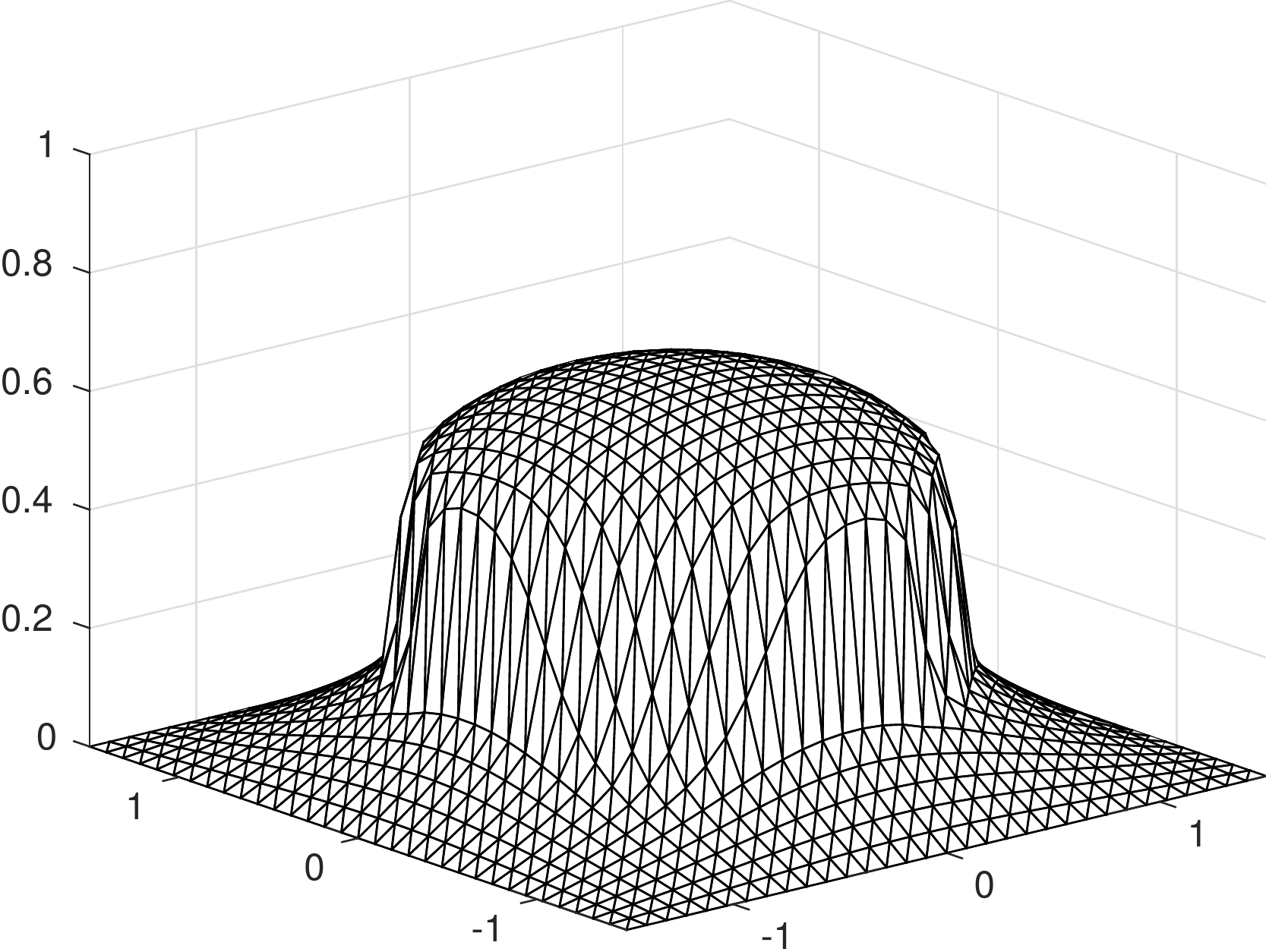} \hspace*{2mm}
\includegraphics[width=.3\linewidth]{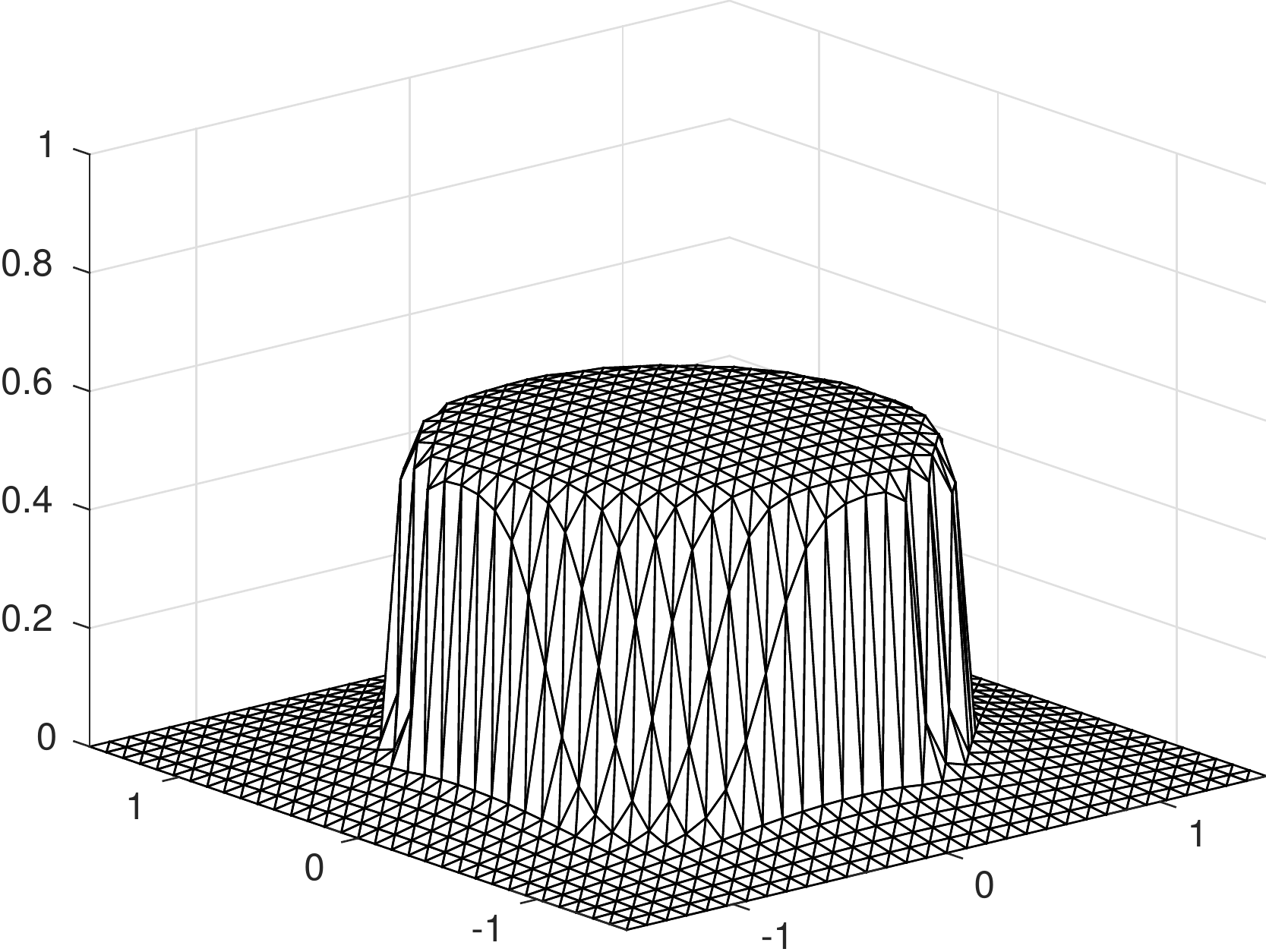}  \hspace*{2mm}
\includegraphics[width=.3\linewidth]{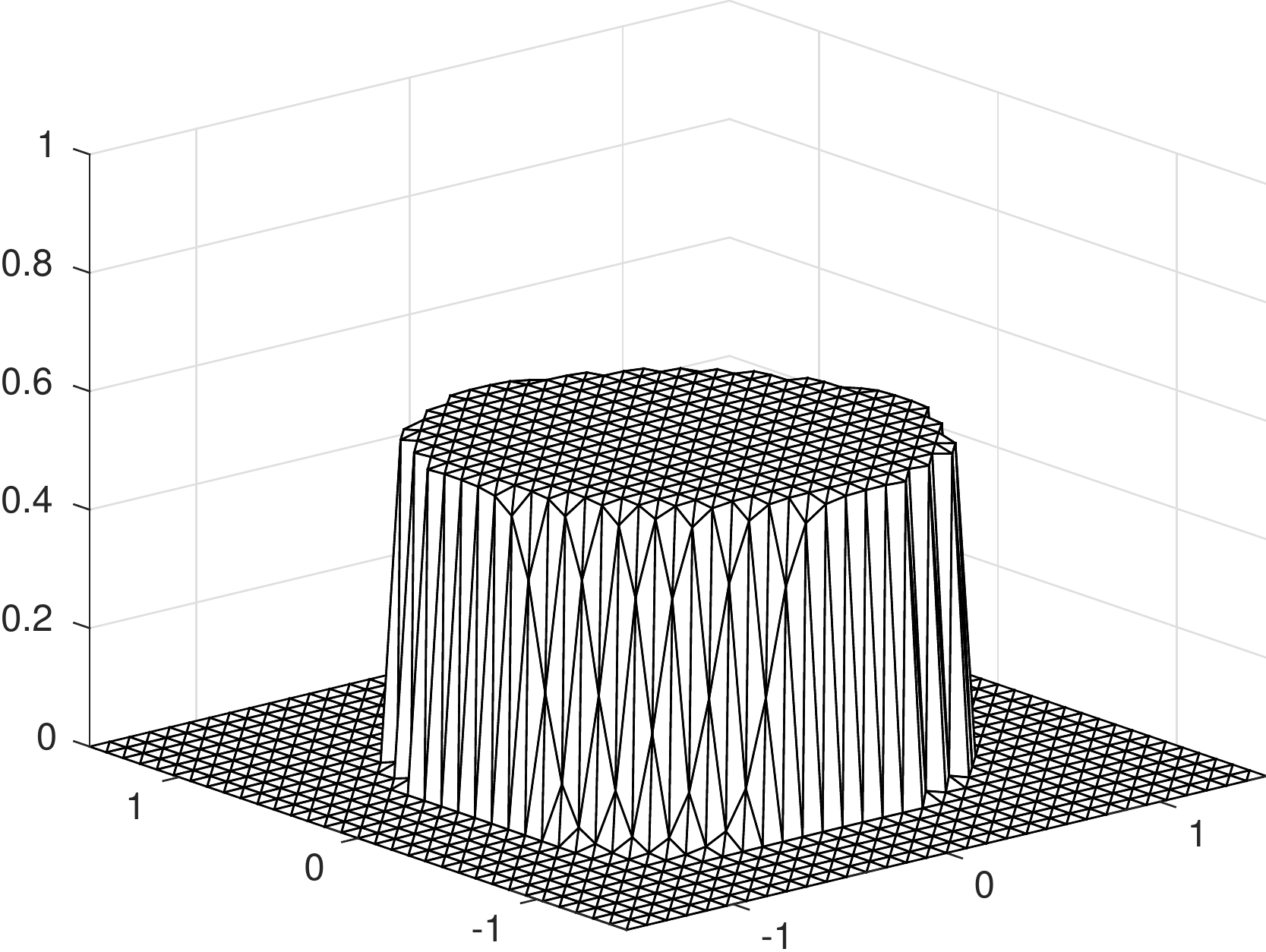} 
\caption{\label{fig:ex_1_comp_reg} Numerical approximations at $t\approx 0.2$ for
$\veps = h^\a$, $\a=1/2,1,2$ (left to right) in Example~\ref{ex:decr_disk}. 
In comparison with the solution obtained with the implicit scheme 
shown in Figure~\ref{fig:decr_disk} we observe a smoothing of the discontinuity.}
\end{figure}


\begin{figure}[p]
\includegraphics[width=.48\linewidth]{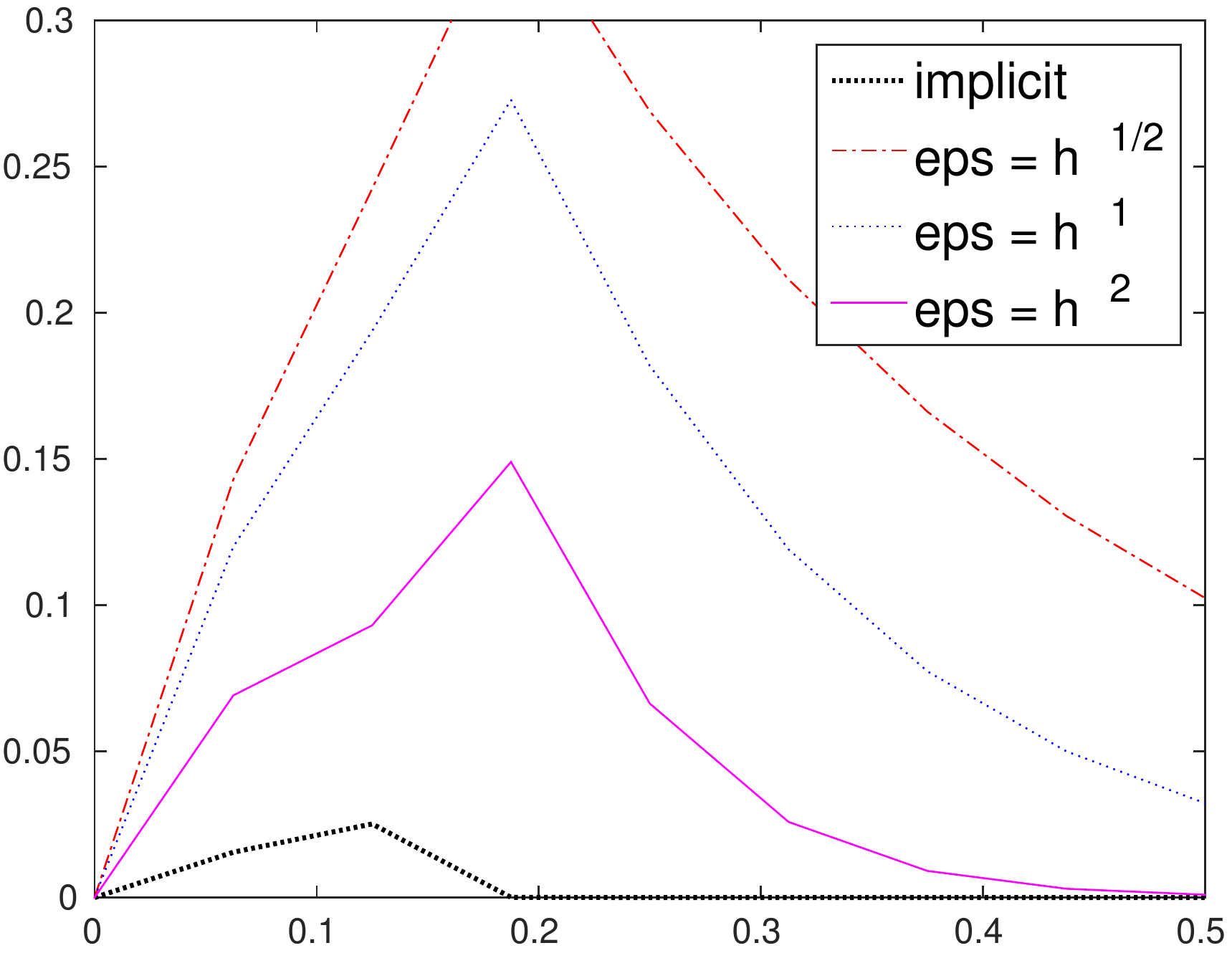} \hspace*{2mm}
\includegraphics[width=.48\linewidth]{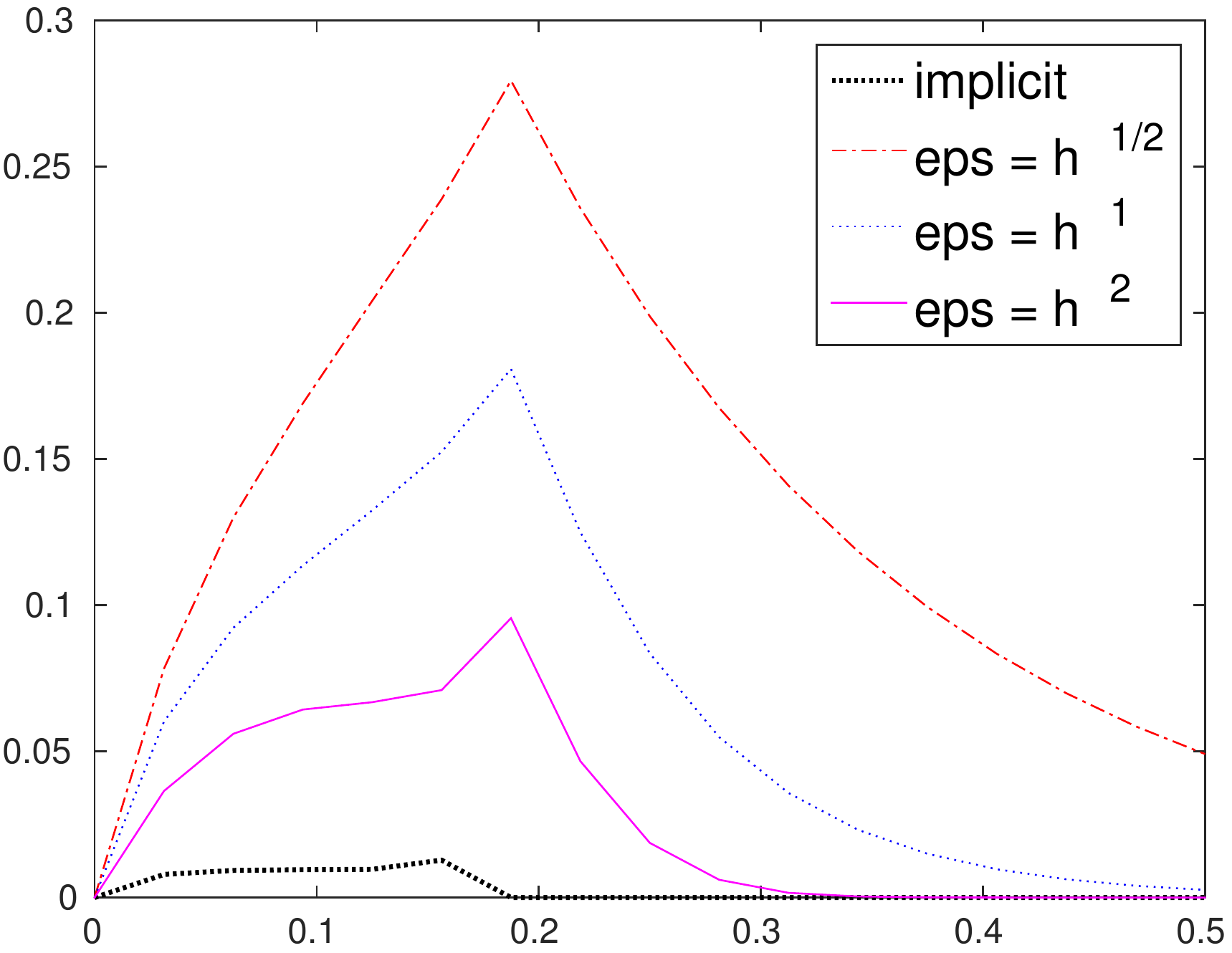} 
\caption{\label{fig:ex_2_l2_errs_time} $L^2$ errors as functions of
$t\in [0,1]$ in Example~\ref{ex:decr_cone} for the semi-implicit scheme with
$\veps = h^\a$, $\a=1/2,1,2$, and implicit approximations on triangulations
$\cT_\ell$, $\ell=4$ (left) and $\ell=5$ (right).}
\end{figure}

\begin{figure}[p]
\begin{minipage}{.55\linewidth}
{\small \begin{tabular}{|c|c|c|c|c|} \hline
$\ell$ & implicit & ${\veps = h^{1/2}}$ & ${\veps = h}$ & ${\veps = h^2}$ \\\hline\hline  
3 & 0.1100 & 0.3368 & 0.2936 & 0.1809 \\\hline
4 & 0.0753 & 0.3432 & 0.2729 & 0.1490 \\\hline
5 & 0.0129 & 0.2795 & 0.1808 & 0.0956 \\\hline
6 & 0.0066 & 0.2169 & 0.1087 & 0.0588 \\\hline
7 & --     & 0.1615 & 0.0617 & 0.0364 \\\hline
8 & --     & 0.1161 & 0.0341 & 0.0241 \\\hline
9 & --     & 0.0814 & 0.0186 & 0.0149 \\\hline
10 & --    & 0.0585 & 0.0101 & 0.0089 \\\hline
\end{tabular}} 
\end{minipage}
\begin{minipage}{.44\linewidth}
\includegraphics[width=\linewidth]{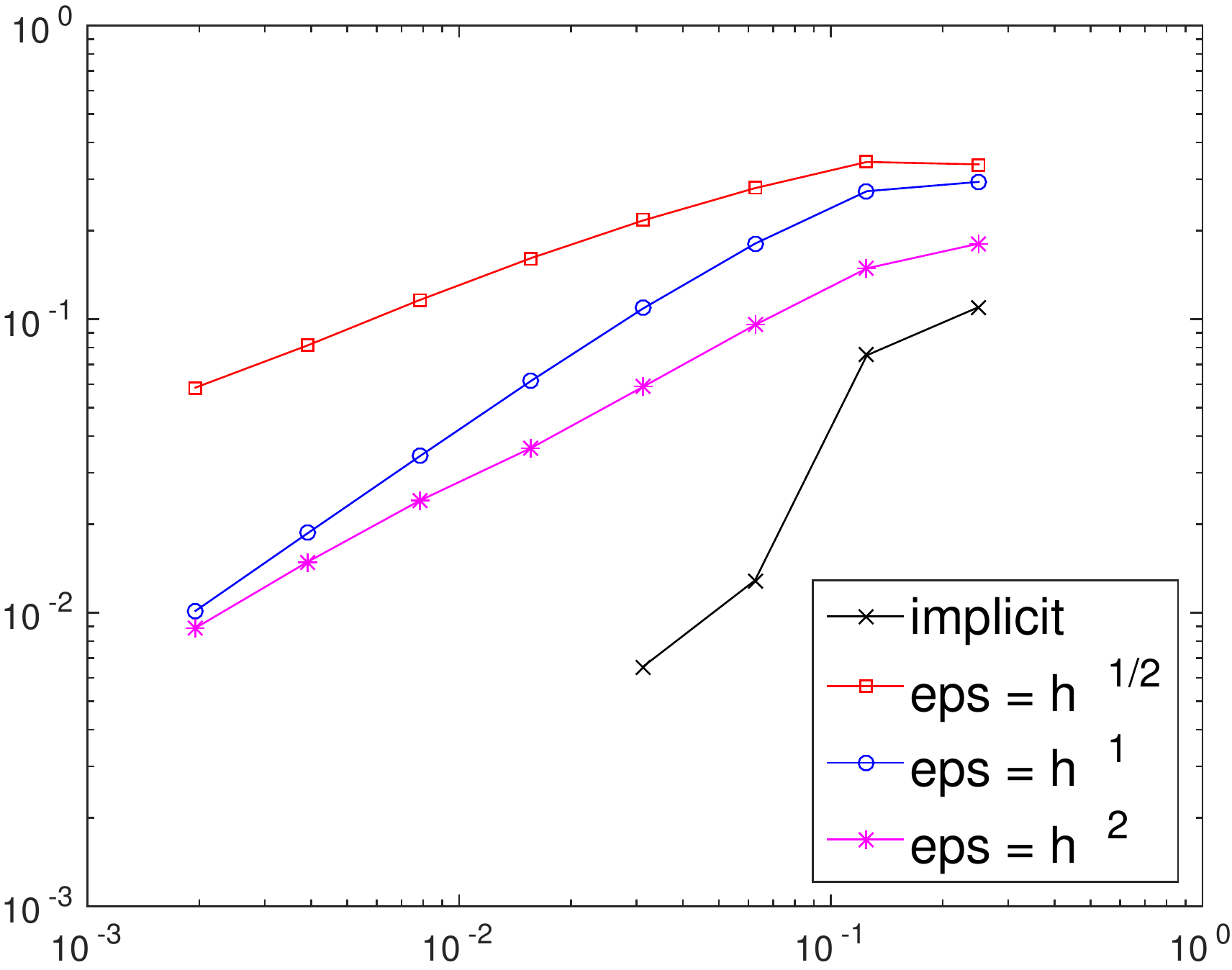}
\end{minipage}
\vspace*{-2mm}
\caption{\label{tab:ex_2_l2_errs} Maximal $L^2$ errors for different choices of 
$\veps$ and on different triangulations~$\cT_\ell$ of level $\ell$
in Example~\ref{ex:decr_cone}.}
\vspace*{2mm}
\end{figure}

\begin{figure}[htb]
\includegraphics[width=.3\linewidth]{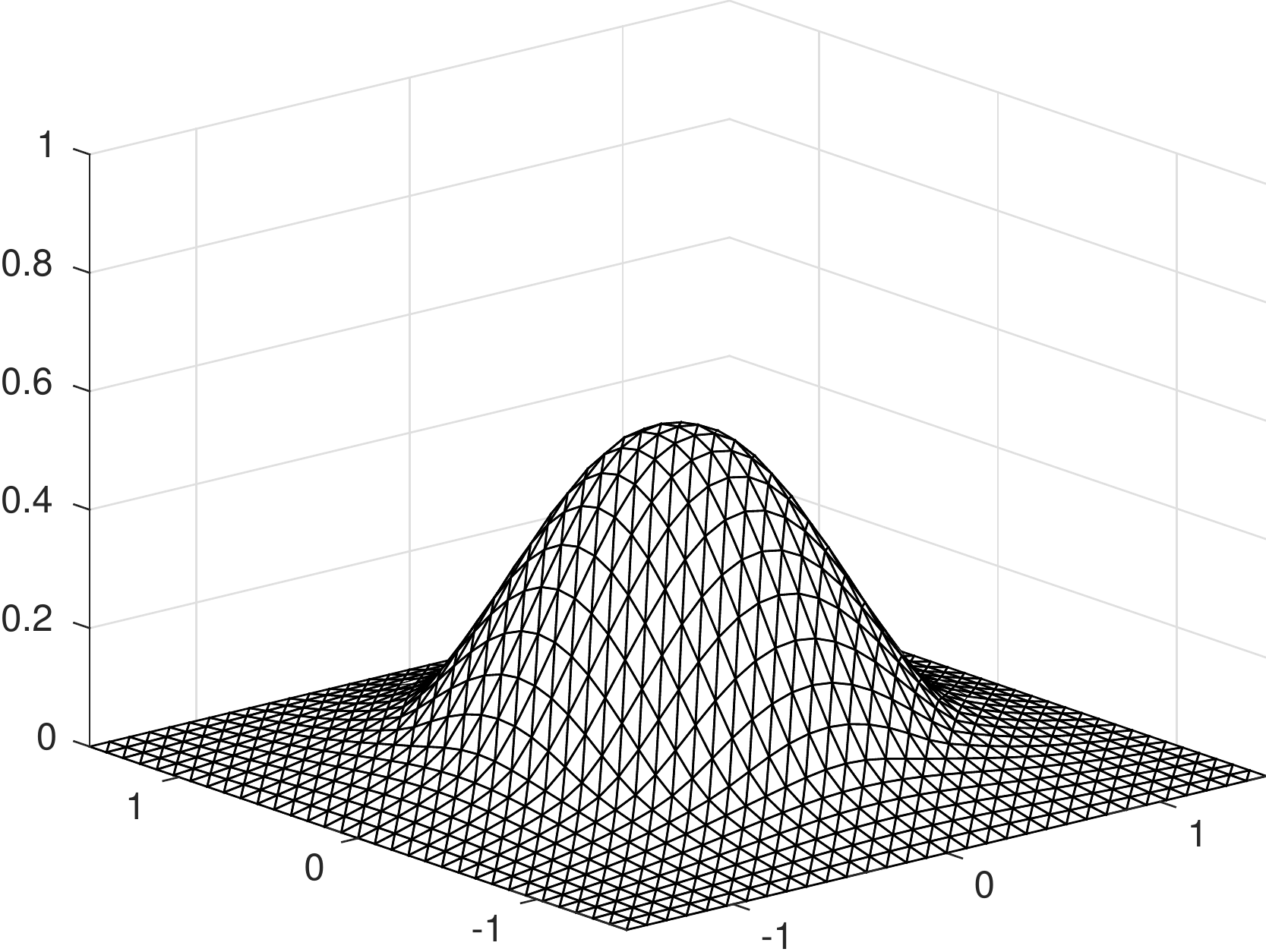} \hspace*{2mm}
\includegraphics[width=.3\linewidth]{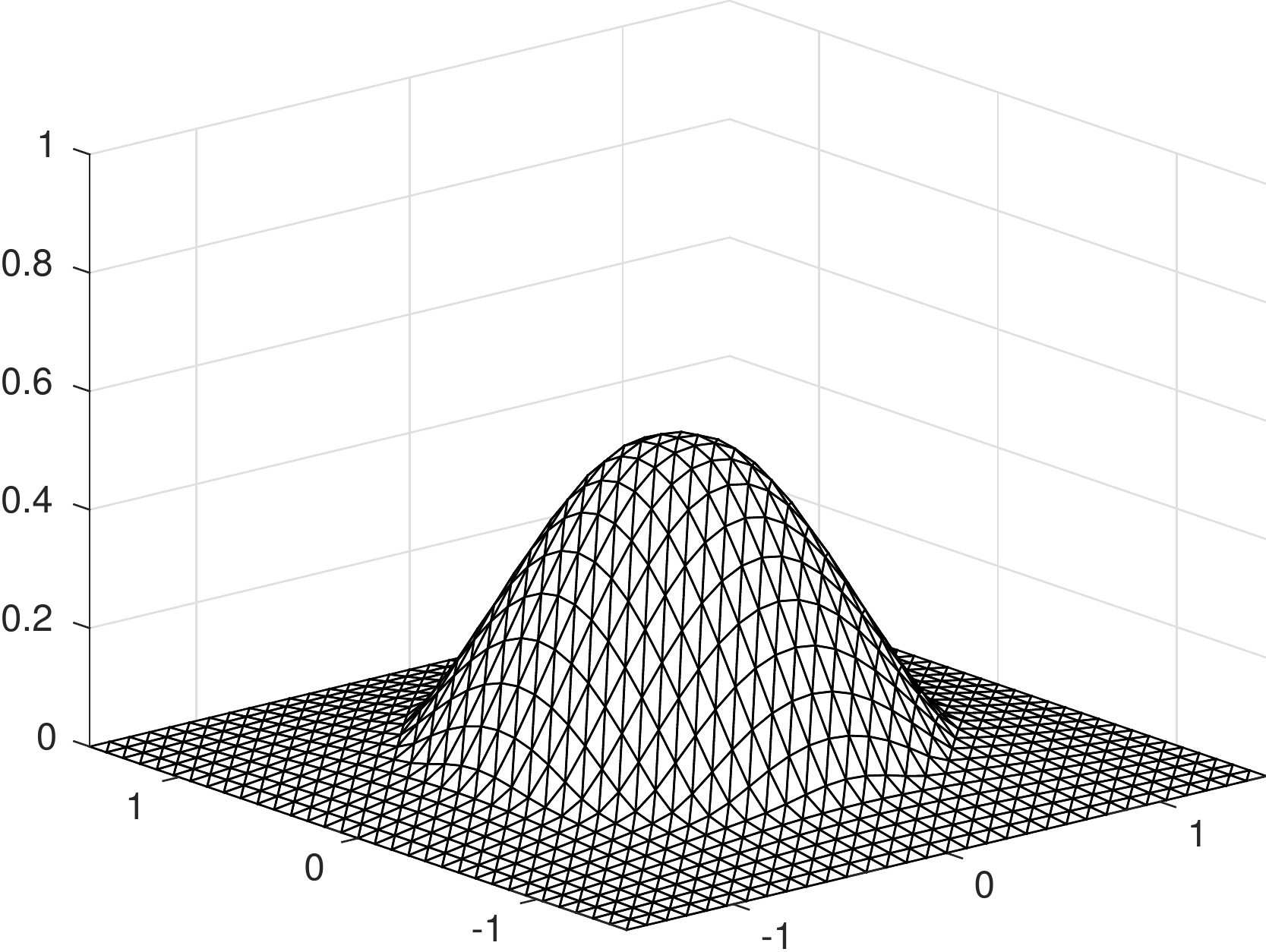} \hspace*{2mm}
\includegraphics[width=.3\linewidth]{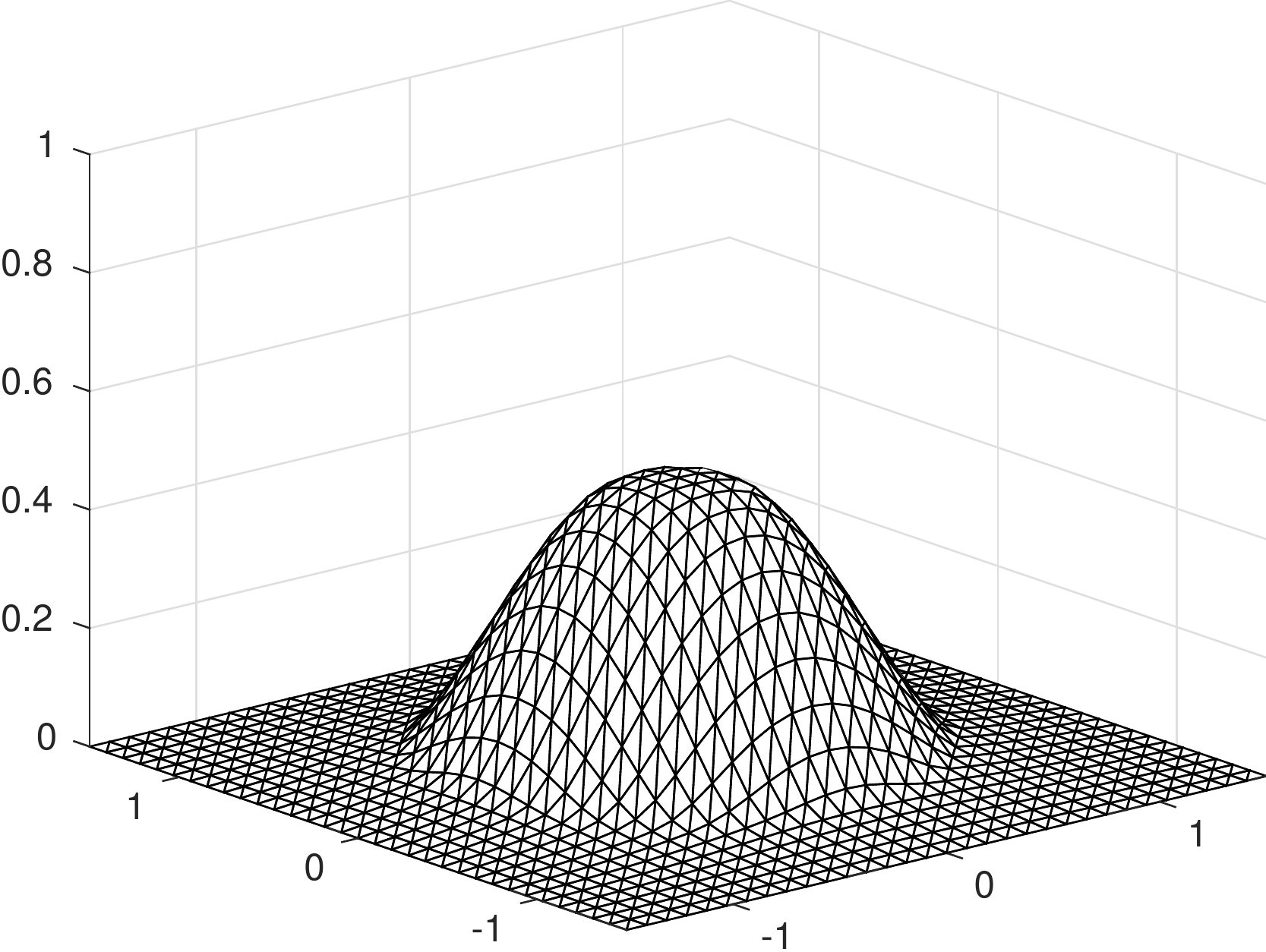} 
\caption{\label{fig:ex_2_comp_reg} Numerical approximations at $t\approx 0.1$ for
$\veps = h^\a$, $\a=1/2,1,2$ (left to right) in Example~\ref{ex:decr_cone}. 
In comparison with the solution obtained with the implicit scheme 
shown in Figure~\ref{fig:decr_cone} we observe a rounding of the kinks.}
\end{figure}


\subsection{Experimental observations}
We computed numerical approximations with implicit and the semi-implicit
schemes on sequences of quasi-uniform triangulations with mesh-size $h$, 
using different regularization parameters $\veps$, and the fixed relation 
$\tau = h/4$. 

\subsubsection{Results for Example~\ref{ex:decr_disk}}
In Figure~\ref{fig:ex_1_l2_errs_time} we plotted the $L^2$ errors
for the implicit and the semi-implicit schemes with regularization
parameters $\veps=h^\a$, $\a=1,1/2,2$, as functions of $t\in [0,T]$, $T=1$, 
obtained for the triangulations $\cT_4$ and $\cT_5$. 
We observe that the $L^2$ errors decrease monotonically with~$\veps$ during most of
the evolution with a certain stagnation, and that the errors obtained with the implicit scheme
are comparable as long as the solution is nontrivial. In particular, the
implicit scheme predicts accurately the extinction time $t=0.5$ in contrast to the
approximations obtained with the regularized, semi-implicit method. 
The maximal $L^2$~errors on $t\in [0,T]$ for several triangulations of
decreasing mesh size displayed in Figure~\ref{tab:ex_1_l2_errs} show that
for a larger value of~$\veps$ we obtain a better experimental convergence rate.
This confirms the critical dependence of our error estimates on the 
regularization parameter $\veps$. No clear experimental convergence rate can be 
deduced for the implicit approach although we used a
stringent stopping criterion (residual
less than $\d_{\rm stop} = h^5$ in the $\ell^2$-norm).
This condition is dictated by theory of the alternating direction
method of multipliers (ADMM) of~\cite{BarMil17-pre}, and
guarantees that the computational results are not due to poor
resolution, but prevents ADMM from converging beyond~6 uniform
refinements, namely for $h\le 5\cdot 10^{-2}$.
Figure~\ref{fig:ex_1_comp_reg} displays snapshots of numerical solutions
on the same triangulation $\cT_5$ but with different regularization parameters
$\veps$ at $t\approx 0.2$. As expected, the smearing effect across
the jump discontinuity set of the exact solution depends on $\veps$.
The choice $\veps = h^2$ appears to give very accurate approximations on $\cT_5$
although, as depicted in Figure~\ref{tab:ex_1_l2_errs}, it exhibits
the worse experimental convergence rate.

\subsubsection{Results for Example~\ref{ex:decr_cone}}
The results of our numerical experiments shown in 
Figures~\ref{fig:ex_2_l2_errs_time}, \ref{tab:ex_2_l2_errs}, 
and~\ref{fig:ex_2_comp_reg} are similar to those for Example~\ref{ex:decr_disk}. 
Here, we observe the best experimental convergence rate for the
choice $\veps=h$ instead of $\veps = h^{1/2}$ which may be explained
by the uniform Lipschitz continuity of the solution. The implicit treatment
leads to smaller approximation errors but, as in
Example~\ref{ex:decr_disk}, the stringent stopping criterion for
ADMM prevents its convergence beyond six uniform mesh
refinements.

\subsubsection{Conclusions}
Our numerical experiments confirm that the error estimates
for the semi-implicit scheme depend on the inverse of the regularization
parameter $\veps$. The experimental convergence rates
are better than those predicted by theory: for $\tau$
proportional to $h$ we expect no convergence (see Corollary \ref{thm:tv_flow_simpl}).
This feature appears to be related to special
regularity properties of the explicit solutions such as $\p_t u(t) \in L^\infty(\O)$
and $u(t)\in W^{1,\infty}(\O)$ for all $t\in (0,T)$ 
in Examples~\ref{ex:decr_disk} and~\ref{ex:decr_cone}, respectively.
The implicit scheme leads to highly accurate approximations that 
provide good predictions of extinction times, but 
require a substantially larger computational effort. In fact, finding 
reliable stopping criteria for the iterative solver, the alternating direction
method of multipliers, is a challenging task. Therefore, the semi-implicit scheme
may also be applied as iterative solver for each time step of
the implicit scheme.

\bigskip 
{\em Acknowledgments.} SB and RHN acknowledge hospitality  
of the Hausdorff Research Institute for Mathematics within the trimester program 
{\em Multiscale Problems: Algorithms, Numerical Analysis and Computation}.
RHN was partially supported as Simons Visiting Professor, in connection 
with the Oberwolfach Workshop {\em Emerging Developments in Interfaces and
Free Boundaries}, as well as by the NSF grant DMS-1411808. SB also
acknowledges support by the DFG priority programme SPP-1962. 

\bibliographystyle{amsalpha}
\bibliography{refs}

\end{document}